\documentclass{article}
\usepackage [utf8] {inputenc}
\usepackage[T1]{fontenc}
\usepackage{mathtools}
\usepackage{amsthm}
\usepackage{amssymb}
\usepackage{enumerate}
\usepackage{float}
\usepackage[ruled]{algorithm2e}
\usepackage{hyperref}

\newtheorem{theorem}{Theorem}[section]
\newtheorem{proposition}[theorem]{Proposition}
\newtheorem{corollary}[theorem]{Corollary}
\newtheorem{lemma}[theorem]{Lemma}

\DeclareMathOperator{\Sym}{Sym}

\DeclareMathOperator{\Aut}{Aut}
\DeclareMathOperator{\AGL}{A\Gamma L}
\DeclareMathOperator{\GamL}{\Gamma L}
\DeclareMathOperator{\GamO}{\Gamma O}
\DeclareMathOperator{\GL}{GL}
\DeclareMathOperator{\SL}{SL}

\DeclareMathOperator{\SU}{SU}

\DeclareMathOperator{\Sz}{Sz}

\DeclareMathOperator{\GF}{GF}
\DeclareMathOperator{\VO}{VO}
\DeclareMathOperator{\VSz}{VSz}
\DeclareMathOperator{\VD}{VD}

\title{Two-closure of rank~3 groups in polynomial time}
\author{Saveliy V. Skresanov}
\date{\footnotesize Sobolev Institute of Mathematics, Novosibirsk, Russia.\\ E-mail: skresan@math.nsc.ru}

\begin{document}

\maketitle
\begin{abstract}
	A finite permutation group \( G \) on \( \Omega \) is called
	a \textit{rank~\( 3 \) group} if it has precisely three orbits in its induced action
	on \( \Omega \times \Omega \). The largest permutation group on \( \Omega \)
	having the same orbits as \( G \) on \( \Omega \times \Omega \) is called
	the \textit{\( 2 \)-closure} of \( G \). We construct a polynomial-time algorithm
	which given generators of a rank~3 group computes generators of its 2-closure.
\end{abstract}

\section{Introduction}

The graph isomorphism problem is one of the central topics in computational complexity theory,
being one of the natural candidates for a problem with an intermediate complexity status.
In its original formulation the problem asks if it is possible to determine in polynomial time that two graphs
given by their adjacency matrices are isomorphic or not, but it is well known to be equivalent to the
problem of computing generators of the full automorphism group of a given graph.

As the exact complexity status of graph isomorphism remains unknown, several successful attacks on the problem
have been made. Luks~\cite{luksBounded} developed a polynomial-time algorithm for isomorphism of graphs of bounded valence,
Babai and Luks~\cite{babaiCanonic} presented the moderately exponential \( \exp(O(\sqrt{n}\log^2 n)) \) algorithm
for graphs with \( n \) vertices, and finally in 2017 Babai solved graph isomorphism in
quasi-polynomial time, i.e.\ in time \( \exp(O(\log^c n)) \), for some \( c > 1 \)~\cite{babaiGI}.

Owing to the immense difficulty of the general problem, several relaxations have been developed,
for instance, graph isomorphism problem for graphs in specific classes such as planar graphs~\cite{hopcroftPlanar},
graphs of bounded Hadwiger number~\cite{ponomHadwiger} or bounded treewidth~\cite{bodlaenderWidth}.
Another approach is inspired by group theory and is chosen in this work.
Suppose we are given enough many automorphisms of a graph, is it then possible to reconstruct
its full automorphism group in polynomial time? One way to define ``enough'' would be to require known automorphisms to
act transitively on vertices, arcs and non-arcs of our graph; note that in this case the graph is strongly regular
and it is called a \textit{rank~\( 3 \) graph}. More generally, if \( G \) is a permutation group on \( \Omega \),
its orbits on \( \Omega \times \Omega \) are called \textit{\( 2 \)-orbits} and the number of 2-orbits is the \textit{rank}
of \( G \). If \( G \) is a permutation group of rank~3, it has one diagonal and two non-diagonal orbits, and in the
case when the order of \( G \) is even, each non-diagonal orbit induces a rank~3 graph.

If \( G \) is an arbitrary permutation group (of any rank), the largest permutation group on \( \Omega \) having
the same 2-orbits as \( G \) is called the \textit{\( 2 \)-closure} of \( G \) and denoted by \( G^{(2)} \).
Clearly, the full automorphism group of a rank~3 graph associated to a rank~3 group \( G \) is precisely \( G^{(2)} \),
so now our relaxation of graph isomorphism can be stated as ``given a rank~3 graph and a list of generators of an associated
rank~3 group \( G \), compute \( G^{(2)} \)''. This is the problem we solve in this paper, in fact, we even
drop the requirement of \( G \) being associated with some graph:

\begin{theorem}\label{mainth}
	Let \( G \) be a finite permutation group of rank~\( 3 \).
	Then one can compute the \( 2 \)-closure \( G^{(2)} \) in polynomial time.
\end{theorem}

Throughout this paper permutation groups will be specified by lists of generating permutations,
and ``polynomial-time'' means polynomial in the degree of the groups involved. 

The reader may validly point out that the measure of ``enough automorphisms'' we choose in this paper
is too restrictive, indeed, rank~3 groups admit a detailed classification~\cite{bannai, kantor, liebeckRank3, liebeckAffine}.
Nevertheless, this approach is quite fruitful. It turns out that the set of 2-orbits of an arbitrary permutation group
forms a \textit{coherent configuration}, a combinatorial object closely related to the original graph isomorphism
problem through the Weisfeiler-Leman color refinement algorithm~\cite{weisfeiler}. Such a coherent configuration is essentially
a colored graph, and the full automorphism group of this graph is the 2-closure of the original group.
This point of view initiated by Ponomarenko in~\cite{ponomGI2Cl} led to a series of results computing 2-closures
of permutation groups from various classes, such as, groups of odd order~\cite{evdokimovOdd}, nilpotent groups~\cite{ponomGI2Cl},
\( \frac{3}{2} \)-transitive groups~\cite{vasilev32} and supersolvable groups~\cite{ponomVasSuper}.
Our contribution can be thought as a continuation of this line of research.

A common approach to the computational 2-closure problem consists in embedding the 2-closure
inside a group with restricted composition factors (for example, an iterated wreath product of cyclic groups) and
then applying the Babai-Luks algorithm~\cite[Corollary~3.6]{babaiCanonic} to find the 2-closure in this larger group.
Unfortunately one cannot put a sufficient restriction on the composition factors of rank~3 groups
as arbitrarily large alternating and linear groups may occur. Therefore here we utilize a different approach,
namely, in the most of cases we constructively recognize a graph associated to the given rank 3 group.
For instance, the affine polar graph \( \VO^{\epsilon}_{2m}(q) \) is determined by the order \( q \) of the field,
dimension of the space \( 2m \) and a quadratic form of type \( \epsilon = \pm \), so if we have successfully
identified vertices of our graph with corresponding vectors over \( \GF(q) \) and reconstructed the associated quadratic form,
writing down the full automorphism group becomes a relatively easy task.

The structure of the paper is as follows. In Section~\ref{subsRank3} we review the description of 2-closures
of rank~3 groups and the classification of affine rank~3 groups, while in Section~\ref{subsPtools} we list the fundamental
polynomial-time algorithms used in this work. Proof of the main result is split into four parts: in Section~\ref{subsNonaff}
we deal with nonaffine rank~3 groups and in Section~\ref{subsAffcase}, consisting of Subsections~\ref{subsubsSmallcase},
\ref{subsubsTensor} and \ref{subsubsQuadratic}, we cover affine groups. At the end of each part we provide the relevant
polynomial-time procedure in terms of pseudocode (Algorithms~\ref{nonaffAlgo}--\ref{qformAlgo}).

\section{Preliminaries}\label{secPrelim}

% TODO: explain what is "subdegrees"?
All groups are assumed to be finite. 

Our notation is mostly standard. If \( G \leq \Sym(\Delta) \) and \( H \leq \Sym(X) \)
we write \( G \wr H \leq \Sym(\Delta \times X) \) for the imprimitive wreath product of \( G \) and \( H \),
and \( G \uparrow H \leq \Sym(\Delta^X) \) for the primitive wreath product. The full semilinear group is denoted
by \( \GamL_a(q) \), and the full affine semilinear group by \( \AGL_a(q) \).
For basic properties of 2-closures the reader is referred to~\cite{wielandt2Cl}.

In the next two sections we review some information on rank~3 groups and polynomial-time
algorithms for permutation groups.

\subsection{Rank 3 groups}\label{subsRank3}

Primitive rank~3 groups were completely classified. A primitive rank~3 permutation
group either stabilizes a nontrivial product decomposition, is almost simple or is an affine group.
The description of rank~3 groups stabilizing a nontrivial product decomposition follows
from the classification of 2-transitive almost simple groups, see Theorem~4.1~(ii)(a) and
Section~5 in~\cite{cameronFinsimp}. Almost simple rank~3 groups were determined in~\cite{bannai}
when the socle is an alternating group, in~\cite{kantor} when the socle is a classical group
and in~\cite{liebeckRank3} when the socle is an exceptional or sporadic group.
The classification of affine rank~3 groups was completed in~\cite{liebeckAffine}.

A description of 2-closures of rank~3 groups was given in~\cite{skresRank3},
the key tool being the aforementioned classification of rank~3 groups. Here we record the main statement.

\begin{theorem}[\cite{skresRank3}]\label{class}
	Let \( G \) be a rank~\( 3 \) permutation group on a set \( \Omega \)
	and suppose that \( |\Omega| > 3^{12} \).
	Then exactly one of the following is true.
	\begin{enumerate}[\normalfont(i)]
		\item \( G \) is imprimitive, i.e.\ it preserves a nontrivial decomposition
			\( \Omega \simeq \Delta \times X \).
			Then \( G^{(2)} = \Sym(\Delta) \wr \Sym(X) \).
		\item \( G \) is primitive and preserves a product decomposition
			\( \Omega \simeq \Delta^2 \) (or in other words, the Hamming graph \( H(2, |\Delta|) \)).
			Then \( G^{(2)} = \Sym(\Delta) \uparrow \Sym(2) \).
		\item \( G \) is primitive almost simple with socle \( L \), i.e.\ \( L \unlhd G \leq \Aut(L) \).
			Then \( G^{(2)} = N_{\Sym(\Omega)}(L) \), and \( G^{(2)} \) is almost simple with socle \( L \).
		\item \( G \) is a primitive affine group which does not stabilize a product decomposition.
			Then \( G^{(2)} \) is also an affine group.
			More precisely, there exist an integer \( a \geq 1 \) and a prime power~\( q \) such that \( G \leq \AGL_a(q) \),
			and exactly one of the following holds (setting \( F = \GF(q) \)).
			\begin{enumerate}[\normalfont(a)]
				\item \( G \leq \AGL_1(q) \). Then \( G^{(2)} \leq \AGL_1(q) \).
				\item \( G \leq \AGL_{2m}(q) \) preserves the bilinear forms graph \( H_q(2, m) \), \( m \geq 3 \). Then
					\( G^{(2)} \leq \AGL_{2m}(q) \) and
					\[ G^{(2)} = F^{2m} \rtimes ((\GL_2(q) \circ \GL_m(q)) \rtimes \Aut(F)). \]
				\item \( G \leq \AGL_{2m}(q) \) preserves the affine polar graph \( \VO_{2m}^{\epsilon}(q) \),
					\( m \geq 2 \), \( \epsilon = \pm \). Then \( G^{(2)} \leq \AGL_{2m}(q) \) and
					\[ G^{(2)} = F^{2m} \rtimes \operatorname{\Gamma O}^{\epsilon}_{2m}(q). \]
				\item \( G \leq \AGL_{10}(q) \) preserves the alternating forms graph \( A(5, q) \). Then
					\( G^{(2)} \leq \AGL_{10}(q) \) and
					\[ G^{(2)} = F^{10} \rtimes ((\GamL_5(q) / \{ \pm 1 \}) \times (F^{\times} / (F^{\times})^2)). \]
				\item \( G \leq \AGL_{16}(q) \) preserves the affine half spin graph \( \VD_{5,5}(q) \).
					Then \( G^{(2)} \leq \AGL_{16}(q) \) and
					\[ G^{(2)} = F^{16} \rtimes ((F^{\times} \circ \operatorname{Inndiag}(D_5(q))) \rtimes \Aut(F)). \]
				\item \( G \leq \AGL_4(q) \) preserves the Suzuki-Tits ovoid graph \( \VSz(q) \),
					\( q = 2^{2e+1} \), \( e \geq 1 \). Then
					\( G^{(2)} \leq \AGL_4(q) \) and
					\[ G^{(2)} = F^4 \rtimes ((F^{\times} \times \Sz(q)) \rtimes \Aut(F)). \]
			\end{enumerate}
	\end{enumerate}
\end{theorem}

Construction and properties of mentioned graphs can be looked up in, for example,~\cite{srgw}, but
in what follows we will most closely work with bilinear forms graphs and affine polar graphs.
We remind that if an affine group \( G = V \rtimes G_0 \) from~(iv) preserves the bilinear forms graph \( H_q(2, m) \),
then the underlying vector space decomposes into a tensor product \( V = U \otimes W \) over \( \GF(q) \),
where \( \dim U = 2 \) and \( \dim W = m \), and \( G_0 \) preserves this decomposition. If \( G \) preserves the
affine polar graph \( \VO_{2m}^{\epsilon}(q) \), then \( V \simeq \GF(q)^{2m} \) can be endowed with a non-degenerate quadratic form
\( \kappa : V \to \GF(q) \) of type \( \epsilon \), such that \( G_0 \) acts on \( \kappa \) by semisimiliarities.

We will also require information on the structure of affine rank~3 groups.

\begin{theorem}[\cite{liebeckAffine}]\label{affrankthree}
	Let \( G \) be a finite primitive affine permutation group of rank~\( 3 \)
	and degree \( n = p^d \), with socle \( V \simeq \GF(p)^d \)
	for some prime~\( p \), and let \( G_0 \) be the stabilizer of the zero vector in \( V \).
	Then \( G_0 \) belongs to one of the following classes.
	\begin{enumerate}
		\item[\normalfont\textbf{(A)}] Infinite classes. These are:
			\begin{enumerate}[\normalfont(1)]
				\item \( G_0 \leq \GamL_1(p^d) \);
				\item \( G_0 \) is imprimitive as a linear group;
				\item \( G_0 \) stabilizes the decomposition of \( V \simeq \GF(q)^{2m} \)
					into \( V = V_1 \otimes V_2 \), where \( p^d = q^{2m} \), \( \dim V_1 = 2 \) and \( \dim V_2 = m \).
					Moreover, \( \SL(V_2) \unlhd G_0^{V_2} \) or \( p = q = 2 \), \( d = a = 8 \),
					or \( \dim V_2 \leq 3 \);
				\item \( G_0 \unrhd \SL_m(\sqrt{q}) \) and \( p^d = q^m \), where \( 2 \) divides \( \frac{d}{m} \);
				\item \( G_0 \unrhd \SL_2(\sqrt[3]{q}) \) and \( p^d = q^2 \), where \( 3 \) divides \( \frac{d}{2} \);
				\item \( G_0 \unrhd \SU_m(q) \) and \( p^d = q^{2m} \);
				\item \( G_0 \unrhd \Omega_{2m}^{\pm}(q) \) and \( p^d = q^{2m} \);
				\item \( G_0 \unrhd \SL_5(q) \) and \( p^d = q^{10} \);
				\item \( G_0 \unrhd B_3(q) \) and \( p^d = q^8 \);
				\item \( G_0 \unrhd D_5(q) \) and \( p^d = q^{16} \);
				\item \( G_0 \unrhd \operatorname{Sz}(q) \) and \( p^d = q^4 \).
			\end{enumerate}
		\item[\normalfont\textbf{(B)}] `Extraspecial' classes.
		\item[\normalfont\textbf{(C)}] `Exceptional' classes.
	\end{enumerate}
	Moreover, classes {\normalfont (B)} and {\normalfont (C)} consist of finitely many groups of degree not exceeding \( 3^{12} \).
\end{theorem}

The correspondence between classes of Theorem~\ref{affrankthree} and rank~3 graphs, as well as formulae for subdegrees
will be given later in Tables~\ref{graphtab} and~\ref{subtab}.

\subsection{Polynomial toolkit}\label{subsPtools}

As was mentioned earlier, both input and output permutation groups used by our algorithms are specified by their generating sets.
In fact, if \( G \) is a permutation group of degree \( n \) generated by a set of elements \( S \), one can in time polynomial
in \( n \) and \( |S| \) construct a set of generators of size at most \( n^2 \), see~\cite[Exercise~4.1]{seress}.
Therefore we can assume that \( |S| \) is bounded polynomially in terms of \( n \), and hence an algorithm
runs in polynomial time if its running time can be bounded as a polynomial in \( n \) only.

Recall that if \( \Sigma \) is a collection of finite simple groups (specified by their names),
\( O_\Sigma(G) \) is the largest normal subgroup of \( G \) such that each composition factor of \( O_\Sigma(G) \)
is isomorphic to a member of \( \Sigma \).
The following fundamental algorithms will be used in the text without further notice.

\begin{proposition}[\cite{seress}]\label{standardAlgos}
	Let \( G \leq \Sym(\Omega) \) be a permutation group given by its generating set.
	The following can be solved in polynomial time:
	\begin{enumerate}[\hspace{1em}\normalfont 1.]
		\item compute the order of \( G \),
		\item given \( g \in \Sym(\Omega) \) test whether \( g \in G \),
		\item compute the orbits of \( G \),
		\item given \( \alpha \) and \( \beta \) in the same orbit, compute \( g \in G \) such that \( \alpha^g = \beta \),
		\item given \( \alpha \in \Omega \) compute the point stabilizer \( G_\alpha \),
		\item compute a minimal block system for \( G \),
		\item compute socle of \( G \), center \( Z(G) \) and derived subgroup \( [G, G] \),
		\item compute \( O_\Sigma(G) \) for any collection of simple groups \( \Sigma \),
		\item test if \( G \) is simple, and if it is, identify its isomorphism type,
		\item enumerate all elements of \( G \) in time polynomial in \( |G| \) and \( |\Omega| \),
		\item all of the above, except items \( 3 \)--\( 6 \), in quotient groups of \( G \).
	\end{enumerate}
\end{proposition}

The last item is based on the work of Kantor and Luks~\cite{kantorQuot}.
For quotient group algorithms we assume that we are given generators of \( G \) and generators of its normal subgroup~\( N \).
Elements of the quotient \( G/N \) are then specified by coset representatives.

We also remark that certain permutation group constructions, such as primitive and imprimitive wreath products,
can be trivially performed in polynomial time.

Next result will be used in the treatment of almost simple groups.
\begin{lemma}[{\cite[Corollary~3.24]{luksPolyNorm}}]\label{polyNorm}
	Let \( T \leq Sym(\Omega) \) be a nonabelian simple group.
	Then \( N_{Sym(\Omega)}(T) \) can be found in polynomial time.
\end{lemma}

Recall that the vertex set of the Hamming graph \( H(a, q) \) is the set of tuples \( (x_1, \dots, x_a) \)
where components \( x_i \) are taken from some finite set of cardinality \( q \), and two vertices are adjacent
if they differ in exactly one place. This labelling of vertices is called a \emph{Hamming labelling},
and the next result provides a polynomial algorithm for finding such a labelling.

\begin{lemma}[\cite{imrichHamming}]\label{hammAlgo}
	Given a graph (specified by a list of vertices and a list of edges), one can decide in polynomial time
	if it is a Hamming graph, and find an appropriate Hamming labelling if that is the case.
\end{lemma}

Affine groups of rank~3 will require certain algorithms for linear groups. If \( G = V \rtimes G_0 \) is a primitive affine group,
observe that \( V \) is the socle and \( G_0 \) is point stabilizer and thus they can be computed in polynomial time.
Since \( |V| = n \) we can avoid most of the theory of effective algorithms for matrix groups, for instance,
we can represent subspaces as lists of vectors, and compute their sums and intersections by straightforward procedures.
One can easily test if \( G_0 \) acts irreducibly on \( V \), indeed, it suffices to check that for every nonzero vector
\( v \in V \) the orbit \( v^{G_0} \) spans \( V \), while the general procedure for matrix groups is highly
nontrivial, see~\cite[Corollary~5.4]{ronyaiFA} and~\cite{parkerMeat}.

Nevertheless we will require the following two matrix groups algorithms.
Note that algorithms run in polynomial time in the size of the input, in particular,
the running time may depend on the size of the matrices involved and the order of the field.
Linear groups are given by their generating matrices, and since all our linear groups
will arise from permutation groups, all parameters (matrix size, order of the field, number of generators, etc.)\ will be under control.

\begin{lemma}[{\cite[Theorem~2]{chistovConj}}]\label{conjalgo}
	Let \( g_1, \dots, g_k \) and \( h_1, \dots, h_k \) be elements of \( \GL_d(p) \), where \( p \) is a prime.
	Then it can be checked in polynomial time whether there
	exists an element \( t \in \GL_d(p) \) such that \( g_i^t = h_i \), \( i = 1, \dots, k \),
	and such an element can be computed if it exists.
\end{lemma}

\begin{lemma}[\cite{ronyai}]\label{ronyaiAlgo}
	Given \( G_0 \leq \GL_d(p) \), where \( p \) is a prime, the centralizer \( C_{\GL_d(p)}(G_0) \) can be computed in polynomial time.
\end{lemma}

\section{Proof of the main result}\label{secProof}

Set \( n = |\Omega| \).  If \( n \leq 3^{12} \), we can find the 2-closure by brute-force,
so from now on we assume that \( n > 3^{12} \). In the following we will often leave out
the condition \( n > 3^{12} \) and write that ``\( n \) is sufficiently large'' instead.

By Theorem~\ref{class}, rank~3 groups fall into four categories: imprimitive groups,
primitive preserving a nontrivial product decomposition, almost simple and affine.
In the next section we provide an algorithm covering the first three cases (Algorithm~\ref{nonaffAlgo}),
while affine groups are subdivided into the ``small'' case (Algorithm~\ref{smallAlgo},
tensor product case (Algorithm~\ref{tensorAlgo}) and quadratic form case (Algorithm~\ref{qformAlgo}).
The main algorithm consists of running Algorithms~\ref{nonaffAlgo} through~\ref{qformAlgo}
and taking the largest output as the final result. Indeed, our group is guaranteed to lie in one
of the four cases mentioned above, and as we will prove below, each algorithm computes the 2-closure
correctly for the respective class.

In the following sections we detail the descriptions of the mentioned classes and provide required algorithms.

\subsection{Nonaffine case}\label{subsNonaff}

In this section we consider imprimitive rank~3 groups, groups preserving a nontrivial product decomposition
and almost simple groups. Notice that despite the name of the section, some affine groups may fall in this case,
for example, affine groups stabilizing a nontrivial product decomposition (i.e.\ linearly imprimitive affine groups).

\begin{proposition}\label{imprimCasePoly}
	It can be tested in polynomial time if \( G \) is an imprimitive rank~\( 3 \) group, and if it is,
	\( G^{(2)} \) can be found in polynomial time.
\end{proposition}
\begin{proof}
	If \( \Delta \) is a nontrivial block of imprimitivity of \( G \), and \( \Omega \) is identified
	with \( \Delta \times X \) for some set \( X \), then \( G^{(2)} = \Sym(\Delta) \wr \Sym(X) \)
	by Theorem~\ref{class}~(i). A nontrivial block system of a permutation group can be found in polynomial time,
	so generators of the required wreath product can be easily obtained. The claim is proved.
\end{proof}

\begin{proposition}\label{prodPoly}
	It can be tested in polynomial time if \( G \) is a primitive rank~\( 3 \) permutation group
	preserving a nontrivial product decomposition, and if it is, \( G^{(2)} \) can be found in polynomial time.
\end{proposition}
\begin{proof}
	Suppose that \( G \) does preserve a nontrivial product decomposition, in particular,
	its subdegrees are \( 2(\sqrt{n}-1) \) and \( (\sqrt{n}-1)^2 \) (these are subdegrees
	of the Hamming graph \( H(2, \sqrt{n}) \)).
	Since the degree of \( G \) is large enough, we can assume that subdegrees of \( G \)
	have different sizes. The smallest nondiagonal 2-orbit induces a Hamming graph on \( \Omega \),
	and Hamming graphs can be recognized and appropriately relabeled in polynomial time
	by Lemma~\ref{hammAlgo}, allowing one to reconstruct the product decomposition \( \Omega \simeq \Delta^2 \)
	preserved by \( G \). We have \( G^{(2)} = \Sym(\Delta) \uparrow \Sym(2) \), and we can obtain
	generators of that wreath product in polynomial time.
\end{proof}

\begin{proposition}\label{affalmsimp}
	It can be tested in polynomial time if \( G \) is a primitive rank~\( 3 \) permutation group
	with nonabelian simple socle, and if it is, \( G^{(2)} \) can be found in polynomial time.
\end{proposition}
\begin{proof}
	One can compute the socle \( S \) and check that it is simple in polynomial time.
	By Theorem~\ref{class}, we have \( G^{(2)} = N_{\Sym(\Omega)}(S) \). The claim now follows from Lemma~\ref{polyNorm}.
\end{proof}

We summarize the algorithm for the nonaffine case in the following figure.
\medskip

\begin{algorithm}[H]
	\SetAlgoLined
	\KwIn{generators of a rank~3 group \( G \leq \Sym(\Omega) \) of sufficiently large degree}
	\KwOut{generators of \( G^{(2)} \) if \( G \) is imprimitive or primitive stabilizing a nontrivial product decomposition
	or almost simple, and failure otherwise}
	\medskip

	\If{\( G \) is imprimitive}{
		Find a nontrivial decomposition \( \Omega \simeq \Delta \times X \)\;
		\Return{\( \Sym(\Delta) \wr \Sym(X) \)}
		}
	\smallskip

	Set \( \Gamma \) to be the smallest 2-orbit of \( G \)\;
	\If(\tcp*[f]{Lemma~\ref{hammAlgo}}){\( \Gamma \) is a Hamming graph}{
		Find a decomposition \( \Omega \simeq \Delta^2 \)\;
		\Return{\( \Sym(\Delta) \uparrow \Sym(2) \)}
		}
	\smallskip

	Compute the socle \( S \) of \( G \)\;
	\If{\( S \) is a nonabelian simple group}{
		\Return{\( N_{\Sym(\Omega)}(S)\)}
		\tcp*[f]{Lemma~\ref{polyNorm}}
		}
	\smallskip

	\tcp{If the group is affine}
	\Return{failure}
	\caption{Nonaffine\label{nonaffAlgo}}
\end{algorithm}

\subsection{Affine case}\label{subsAffcase}

Recall that if \( G \) is an affine rank~3 group, then \( G \leq \AGL_a(q) \) and \( G \) decomposes as \( G = V \rtimes G_0 \),
where \( V \) is an \( a \)-dimensional vector space over \( \GF(q) \), while \( G_0 \leq \GamL_a(q) \).
The smallest value of \( a \) such that \( G \leq \AGL_a(q) \) for some \( q \) will be denoted by~\( a(G) \).
We continue to use our standard notation \( n = |\Omega| = |V| \). Note also that \( n = p^d = q^a \),
where \( q \) is a power of a prime \( p \) and \( a \) divides \( d \).

By Theorem~\ref{affrankthree} affine groups of rank~\( 3 \) fall into three categories (A), (B) and (C),
where the latter two encompass only a finite number of groups and are not considered since \( n > 3^{12} \).
Also notice that groups from class (A2) stabilize a nontrivial product decomposition,
and so they were covered in the previous section. 

Finally, structure of affine rank~3 groups, their subdegrees, graphs preserved and values of \( a(G) \)
can be easily read of~\cite[Tables~1 and~3]{skresRank3}. Here we provide a summary for reader's convenience.
The meaning of the column ``Class'' is explained right after the tables.

\begin{table}[H]
\centering
\caption{Linearly primitive rank 3 group in class (A)}\label{graphtab}
\begin{tabular}{l l l l l}
	\hline
	       & Type of \( G \) & a(G)           & Graph & Class \\
	\hline
	(A1):  & \( G_0 < \Gamma L_1(p^d) \) & \( 1 \) & One-dimensional affine & (S)\\
	(A3):  & tensor product & \( 2m \) & \( H_q(2, m) \) & (T)\\
	(A4):  & \( G_0 \unrhd \SL_m(\sqrt{q}) \) & \( m \) & \( H_{\sqrt{q}}(2, m) \) & (T)\\
	(A5):  & \( G_0 \unrhd \SL_2(\sqrt[3]{q}) \) & \( 2 \) & \( H_{\sqrt[3]{q}}(2, 3) \) & (T)\\
	(A6):  & \( G_0 \unrhd \SU_m(q) \) & \( m \) & \( \VO^{\epsilon}_{2m}(q) \), \( \epsilon = (-1)^m \) & (Q)\\
	(A7):  & \( G_0 \unrhd \Omega^{\epsilon}_{2m}(q) \) & \( 2m \) & \( \VO^{\epsilon}_{2m}(q) \) & (Q)\\
	(A8):  & \( G_0 \unrhd \SL_5(q) \) & \( 10 \) & \( A(5, q) \) & (S)\\
	(A9):  & \( G_0 \unrhd B_3(q) \) & \( 8 \) & \( \VO^+_8(q) \) & (S)\\
	(A10): & \( G_0 \unrhd D_5(q) \) & \( 16 \) & \( \VD_{5,5}(q) \) & (S)\\
	(A11): & \( G_0 \unrhd \operatorname{Sz}(q) \) & \( 4 \) & \( \operatorname{VSz}(q) \) & (S)\\
	\hline
\end{tabular}
\end{table}

\begin{table}[h]
\centering
\caption{Subdegrees of some groups in class (A)}\label{subtab}
\begin{tabular}{l l c}
	\hline
	       & \( n = p^d \) & Subdegrees \\
	\hline
	(A3):  & \( q^{2m} \) & \( (q+1)(q^m - 1) \), \( q(q^m - 1)(q^{m-1} - 1) \)\\
	(A4):  & \( q^m \) & \( (\sqrt{q}+1)(\sqrt{q}^m - 1) \), \( \sqrt{q}(\sqrt{q}^m - 1)(\sqrt{q}^{m-1} - 1) \)\\
	(A5):  & \( q^2 \) & \( (\sqrt[3]{q}+1)(q - 1) \), \( \sqrt[3]{q}(q - 1)(\sqrt[3]{q}^2 - 1) \)\\
	(A6):  & \( q^{2m} \) & \(
	       \begin{cases}
		       (q^m - 1)(q^{m-1} + 1), q^{m-1}(q-1)(q^m - 1), \, m \text{ even}\\
		       (q^m + 1)(q^{m-1} - 1), q^{m-1}(q-1)(q^m + 1), \, m \text{ odd} 
	       \end{cases} \)\\
	(A7):  & \( q^{2m} \) & \(
	       \begin{cases}
		       (q^m - 1)(q^{m-1} + 1), q^{m-1}(q-1)(q^m - 1), \, \epsilon = + \\
		       (q^m + 1)(q^{m-1} - 1), q^{m-1}(q-1)(q^m + 1), \, \epsilon = - 
	       \end{cases} \)\\
	\hline
\end{tabular}
\end{table}

For the algorithmic considerations to follow it is more useful to regroup linearly primitive
affine rank~3 groups into the following three families:
\begin{enumerate}
	\item[\textbf{(S)}] Small groups, that is, groups with \( a(G) \leq 16 \) and
		such that the zero stabilizer \( G_0 \) is 4-generated.
	\item[\textbf{(T)}] Tensor case, that is, classes (A3), (A4) and (A5).
		These groups stabilize a nontrivial tensor decomposition.
	\item[\textbf{(Q)}] Quadratic form case, that is, classes (A6) and (A7). These groups preserve a quadratic form.
\end{enumerate}

Classes (A3)--(A7) fall into cases (T) and (Q) by definition, 
so it is left to sort out classes (A1) and (A8)--(A11). We need the following observation.

\begin{lemma}[{\cite[Corollary to Theorem~1]{voltaGens}}]\label{almsimpGens}
	Let \( G \) be an almost simple group with socle \( S \).
	Then \( G/S \) and \( G \) are 3-generated.
\end{lemma}

\begin{proposition}\label{ains}
	Classes (A1), (A8)--(A11) lie in (S).
\end{proposition}
\begin{proof}
	Obvious for (A1). If \( G \) lies in (A8)--(A11), then \( G_0/Z(G_0) \) is almost simple and \( Z(G_0) \) is cyclic,
	see~\cite[(1.4)]{liebeckAffine} (it also follows from Table~\ref{graphtab} and Theorem~\ref{class}).
	By Lemma~\ref{almsimpGens} the group \( G_0/Z(G_0) \) is 3-generated.
	Since \( Z(G_0) \) is cyclic, \( G_0 \) is 4-generated, as claimed.
	Table~\ref{graphtab} implies \( a(G) \leq 16 \) at last.
\end{proof}

As a corollary from Theorem~\ref{affrankthree} and Proposition~\ref{ains} we yield
\begin{proposition}
	Let \( G \) be a primitive affine rank~\( 3 \) permutation group on \( \Omega \) which does not preserve a nontrivial product decomposition,
	and assume that \( |\Omega| > 3^{12} \). Then \( G \) belongs to class (S), (T) or (Q).
\end{proposition}

Names of classes, corresponding to cases (A1)--(A11), are written in the last column of Table~\ref{graphtab}.
We stress that classes (S), (T) and (Q) may overlap, so a group may belong to several classes.

\subsubsection{Small groups case}\label{subsubsSmallcase}

Let \( G \) lie in (S), and recall that \( G = V \rtimes G_0 \),
where \( V \) is an \( a \)-dimensional vector space over \( \GF(q) \)
and \( G_0 \leq \GamL_a(q) \). Moreover, \( a \leq 16 \) and \( G_0 \) is 4-generated.

In order to compute the 2-closure we first show that \( G^{(2)} \) lies in
\( \AGL_a(q) \) for some \( a \leq 48 \) and \( q \).
Then we explicitly enumerate all groups of the form \( \GamL_a(q) \), \( a \leq 48 \),
containing \( G_0 \), and find the 2-closure by brute-force, i.e.\ by listing all elements
of \( \AGL_a(q) \) and checking which permutations preserve the associated rank~3 graph.

\begin{lemma}
	Let \( G \) be an affine rank~\( 3 \) group of sufficiently large degree, and suppose that \( G \)
	does not stabilize a nontrivial product decomposition. Then \( a(G^{(2)}) \leq 3a(G) \).
\end{lemma}
\begin{proof}
	It follows from Theorem~\ref{class} that
	\begin{equation*}
	\begin{split}
		\Aut(H_q(2, m)) \leq \AGL_{2m}(q),\\
		\Aut(\VO^{\pm}_{2m}(q)) \leq \AGL_{2m}(q),\\
		\Aut(A(5, q)) \leq \AGL_{10}(q),\\
		\Aut(\VD_{5,5}(q)) \leq \AGL_{16}(q),\\
		\Aut(\VSz(q)) \leq \AGL_4(q).
	\end{split}
	\end{equation*}
	Also, Theorem~\ref{class}~(iv)(a) implies that if \( G \leq \AGL_1(q) \) then \( G^{(2)} \leq \AGL_1(q) \).
	Table~\ref{graphtab} lists \( a(G) \) and corresponding rank~3 graphs for affine linearly primitive rank~3 groups \( G \),
	so the inequality \( a(G^{(2)}) \leq 3a(G) \) follows easily by inspection of the table. Equality is attained in case (A5),
	where \( a(G) = 2 \), and \( G^{(2)} \) is the automorphism group of the bilinear forms graph \( H_{\sqrt[3]{q}}(2, 3) \),
	so \( a(G^{(2)}) = 6 \).
\end{proof}

\begin{corollary}\label{asmall}
	If \( G \) belongs to (S), then \( a(G^{(2)}) \leq 48 \).
\end{corollary}

\begin{lemma}\label{aglsize}
	Let \( n = q^a \). Then \( |\GamL_a(q)| \leq n^{a+1} \) and \( |\AGL_a(q)| \leq n^{a+2} \).
\end{lemma}
\begin{proof}
	Observe that
	\[ |\GamL_a(q)| = |\GL_a(q)| \cdot |\Aut(\GF(q))| \leq q^{a^2} \cdot q \leq n^{a+1}. \]
	In particular, \( |\AGL_a(q)| = q^a \cdot |\GamL_a(q)| \leq n^{a+2} \).
\end{proof}
\begin{corollary}
	If \( G \) belongs to (S), then \( |G| \leq n^{18} \).
\end{corollary}

Above corollary justifies the name ``small'' for the class (S), indeed,
groups from this class have order polynomially bounded in terms of \( n \).

\begin{lemma}\label{embed}
	Let \( g_1, \dots, g_k \) be some elements of \( \GL_d(p) \), where \( k \) is a fixed constant,
	and let \( G_0 \) be the group generated by \( g_1, \dots, g_k \).
	It can be checked in polynomial time whether \( G_0 \) is conjugate to a subgroup of \( \GamL_a(q) \leq \GL_d(p) \),
	where \( a \) is a fixed constant, and if it is conjugate, an element \( t \in \GL_d(p) \) such that
	\( G_0^t \leq \GamL_a(q) \) can be found in polynomial time. In fact, all embeddings of \( G_0 \) into
	\( \GamL_a(q) \) can be found in polynomial time.
\end{lemma}
\begin{proof}
	Try all \( k \)-tuples \( h_1, \dots, h_k \) of elements of \( \GamL_a(q) \) (there are at most
	\( n^{k(a+1)} \) of those by Lemma~\ref{aglsize}), and using Lemma~\ref{conjalgo} check if there
	exists an element \( t \in \GL_d(p) \) such that \( g_i^t = h_i \), \( i = 1, \dots, k \).
\end{proof}

Recall that there might be several subgroups of the form \( \GamL_a(q) \) inside \( \GL_d(p) \),
depending on the embedding of the multiplicative group of the field \( \GF(q) \) in \( \GL_d(p) \). All of these subgroups are conjugate
in \( \GL_d(p) \), so if \( \GamL_a(q) \leq \GL_d(p) \) and \( G_0 \) lies in some \( \widetilde{\GamL_a(q)} \),
then \( \widetilde{\GamL_a(q)}^t = \GamL_a(q) \) for some \( t \in \GL_d(p) \), and hence \( G_0^t \leq \GamL_a(q) \).
We can find all such \( t \) in polynomial time in the situation of Lemma~\ref{embed}, so in fact we can
find all subgroups of the form \( \GamL_a(q) \) containing \( G_0 \).

\begin{proposition}
	It can be tested in polynomial time if \( G \) belongs to (S), and if it does,
	\( G^{(2)} \) can be found in polynomial time.
\end{proposition}
\begin{proof}
	Find \( |G| \) and check that \( |G| \leq n^{18} \).
	If that is the case, list all elements of \( G_0 \) in polynomial time,
	and check if any \( 4 \)-tuple of elements generates \( G_0 \) (there are
	at most \( n^{18\cdot 4} \) such tuples). Given a generating set of size \( 4 \),
	by Lemma~\ref{embed} we can check in polynomial time if \( G_0 \) is conjugate to a subgroup of
	\( \GamL_a(q) \), \( a \leq 16 \), hence it is possible to check
	if \( G \) belongs to (S) in polynomial time.

	By Corollary~\ref{asmall}, \( G^{(2)} \leq \AGL_a(q) \), \( a \leq 48 \), in particular,
	\( G_0 \leq (G^{(2)})_0 \leq \GamL_a(q) \), \( a \leq 48 \). By Lemma~\ref{embed},
	we can find all groups \( \GamL_a(q) \), \( a \leq 48 \), such that \( G_0 \leq \GamL_a(q) \)
	in polynomial time. In particular, we can find all groups \( \AGL_a(q) \), \( a \leq 48 \), containing \( G \),
	and since \( |\AGL_a(q)| \) is polynomially bounded in terms of \( n \), we can find the subgroup of
	\( \AGL_a(q) \) preserving the 2-orbits of \( G \) in polynomial time.
	The largest such subgroup will give us the 2-closure of~\( G \), as required.
\end{proof}

In the following figure (Algorithm~\ref{smallAlgo}) we give an overview of the algorithm for class (S).
When checking the condition \( a(G) \leq 16 \) we use an additional variable \( f \)
as a flag: \( f = 0 \) if \( G_0 \) cannot be embedded into \( \GamL_a(q) \), \( a \leq 16 \),
and \( f = 1 \) otherwise. In the last part of the algorithm, intersection \( G^{(2)} \cap \AGL_a(q) \)
is computed by brute-force, i.e.\ all elements of \( \AGL_a(q) \) are listed and it is checked which
elements preserve the 2-orbits of \( G \).

\begin{algorithm}[p]
	\SetAlgoLined
	\KwIn{generators of an affine rank~3 group \( G \) of sufficiently large degree \( n = p^d \)}
	\KwOut{generators of \( G^{(2)} \) if \( G \) lies in (S), and failure otherwise}
	\medskip

	Check that \( G \) is affine and construct \( G_0 \)\;
	\smallskip

	\If{\( |G| > n^{18} \)}{\Return{failure}}
	Compute all elements of \( G \)\;
	\eIf{\( G_0 \) is not 4-generated}{
		\Return{failure}
		}{
		Find 4 generators of \( G_0 \)\ by brute-force;
		}
	\medskip

	\tcp{Check that \( a(G) \leq 16 \)}
	Set \( f = 0 \)\;
	\For{\( a = 1, \dots, 16 \), such that \( a \) divides \( d \)}{
		Set \( q = p^{d/a} \)\;
		\If(\tcp*[f]{Lemma~\ref{embed}}){\( G_0 \) embeds in \( \GamL_a(q) \)}{
			Set \( f = 1 \)\;
			}
		}
	\If{\( f = 0 \)}{
		\Return{failure}
		}
	\medskip

	Set \( L \) to be an empty list\;
	\For{\( a = 1, \dots, 48 \), such that \( a \) divides \( d \)}{
		Set \( q = p^{d/a} \)\;
		\For(\tcp*[f]{Lemma~\ref{embed}}){\( \GamL_a(q) \leq \GL_d(p) \) such that \( G_0 \leq \GamL_a(q) \)}{
			Find \( G^{(2)} \cap \AGL_a(q) \) by brute-force\;
			Add \( G^{(2)} \cap \AGL_a(q) \) to the list \( L \)\;
			}
		}
	\smallskip

	\Return{the largest member of \( L \)}
	\caption{Small groups case\label{smallAlgo}}
\end{algorithm}

\subsubsection{Tensor product case}\label{subsubsTensor}

Let \( G \) lie in (T), and recall that \( G = V \rtimes G_0 \), where
\( V \) is a \( 2m \)-dimensional vector space over \( \GF(q) \) and \( G_0 \leq \GamL_{2m}(q) \), \( m \geq 2 \).
The space decomposes into a tensor product \( V = U \otimes W \) over \( \GF(q) \), where \( \dim U = 2 \),
\( \dim W = m \), and one of the two orbits of \( G_0 \) on nonzero vectors consists of simple tensors,
i.e.\ vectors of the form \( u \otimes w \), \( u \in U \), \( w \in W \).

We mention that if \( G_0 \) lies in (A3), it stabilizes a tensor decomposition by definition,
while in classes (A4) and (A5) it follows from the fact that \( G_0 \) stabilizes a certain subspace.
For example, if \( G_0 \) lies in (A4), it stabilizes a subspace \( U \leq V \)
of dimension \( m \) over \( \GF(\sqrt{q}) \), where \( n = q^m \), and hence \( G_0 \)
stabilizes a nontrivial tensor decomposition \( V \simeq U \otimes \GF(q) \),
where the tensor product is taken over \( \GF(\sqrt{q}) \). Similar reasoning holds for (A5).

Since \( G \) from class (T) stabilizes a bilinear forms graph, Theorem~\ref{class} and \cite[Theorem~9.5.1]{brouwerDRG} imply that
\[ G^{(2)} = \GF(q)^{2m} \rtimes ((\GL_2(q) \circ \GL_m(q)) \rtimes \Aut(F)), \]
when \( m > 2 \), and
\[ G^{(2)} = \GF(q)^4 \rtimes ((\GL_2(q) \circ \GL_2(q)) \rtimes \Aut(F)) \rtimes C_2, \]
when \( m = 2 \), where an additional automorphism of order two flips components of simple tensors.
Notice also that the case of \( m = 2 \) is covered by Theorem~\ref{class}~(c), since \( H_q(2, 2) \) is isomorphic to \( \VO_4^+(q) \).

To compute the 2-closure in polynomial time, we will first obtain an explicit embedding of \( G_0 \)
into \( \GamL_{2m}(q) \). Then we will exhibit a basis of \( V \) of the form \( u_i \otimes w_j \), \( i = 1, 2 \), \( j = 1, \dots, m \),
where \( u_1, u_2 \) is a basis of \( U \) and \( w_1, \dots, w_m \) is a basis of \( W \); we will call
such a basis a \emph{tensor basis}. Finally, given a tensor basis and the action of the field \( \GF(q) \)
it is possible to write down generators of \( G^{(2)} \) explicitly.

We start with several general lemmata which will be used in this and next sections to
construct embeddings of \( G_0 \) into groups of the form \( \GamL_a(q) \).
Recall that the group \( \GamL_a(q) \) is completely determined by the embedding
of the field \( \GF(q) \) in \( \GL_d(p) \). Indeed, if \( F = \GF(q) \), then
\( F^{\times} \leq \GamL_a(q) \) and \( \GamL_a(q) \) is precisely the set of all
\( F \)-semilinear mappings of the vector space \( V \). It is enough to know
a generator of \( F^\times \) to specify the field, hence we write ``the field is given''
if such a generator is given as input to an algorithm, and ``the field is known'' if we
computed it at some stage of an algorithm.

Notice that in what follows matrix groups \( \GL_d(p) \), \( G_0 \) etc.\ are given
by generating permutations, in particular, polynomial time still means polynomial in \( n = p^d \).

\begin{lemma}\label{fieldelem}
	If \( c \in \GL_d(p) \) and \( a \) divides \( d \), then we can check in polynomial time whether
	\( c \) generates a multiplicative group of a finite field \( \GF(p^{d/a}) \).
	Moreover, given \( G_0 \leq \GL_d(p) \), we can check whether \( G_0 \leq \GamL_a(p^{d/a}) \),
	where the respective field is generated by the element \( c \).
\end{lemma}
\begin{proof}
	For the first claim, check that the order of \( c \) equals \( p^{d/a}-1 \), and check
	that the set of matrices \( 0 \), \( c, c^2, \dots, c^{p^{d/a}-1} \)
	is closed under addition (clearly then it satisfies all other field axioms).
	For the second claim, it suffices to check that all generators \( g_1, \dots, g_k \) of \( G_0 \)
	act semilinearly on \( V \simeq \Omega \) viewed as a vector space over
	\( \GF(p^{d/a}) \simeq F = 0 \cup \langle c \rangle \). First, check that \( F \) is \( G_0 \)-invariant,
	i.e.\ \( c^g \in F \) for \( g \in \{ g_1, \dots, g_k \} \).
	Now, notice that for each \( v \in V \), \( f \in F \) and every generator \( g \in \{ g_1, \dots, g_k \} \) we have:
	\[ (f \cdot v)^g = v^{fg} = v^{gf^g} = f^g \cdot v^g, \]
	where \( \cdot \) denotes the action of the field on the vector space and is defined as \( f \cdot v = v^f \)
	for \( v \in V \) and \( f \in \langle c \rangle \), and \( 0 \cdot v = 0 \).
	Therefore \( G \) acts semilinearly if and only if for every generator \( g \in \{ g_1, \dots, g_k \} \)
	and every \( f \in F \) we have \( f^g = f^{p^i} \) for some \( i \) (depending only on \( g \)).
	All checks can be done in polynomial time, so the claim is proved.
\end{proof}

\begin{lemma}\label{centrAlgo}
	Let \( H \leq \GL_d(p) \) and suppose that \( V \) decomposes into a direct sum of
	\( k \) irreducible \( H \)-modules over \( \GF(p) \).
	Then \( |C_{\GL_d(p)}(H)| \leq n^k \) and, in particular,
	all elements of \( C_{\GL_d(p)}(H) \) can be enumerated in time polynomial in~\( n^k \).
\end{lemma}
\begin{proof}
	Take a decomposition of \( V \) into irreducible \( H \)-modules
	\[ V = V_1 \oplus \dots \oplus V_k \]
	and fix arbitrary nonzero vectors \( v_i \in V_i \), \( i = 1, \dots, k \).
	Since \( V_i \) is an irreducible \( H \)-module over \( \GF(p) \), it is generated by \( v_i^H \)
	as an abelian group, so \( V \) is generated by all classes \( v_i^H \), \( i = 1, \dots, k \), as an abelian group.

	Now, let \( c \in C_{\GL_d(p)}(H) \), and set \( u_i = v_i^c \), \( i = 1, \dots, k \).
	Fix arbitrary \( i \in \{ 1, \dots, k \} \). Then, for any \( h \in H \) we have
	\[ (v_i^h)^c = (v_i^c)^h = u_i^h, \]
	and hence the action of \( c \) on all elements of the \( H \)-orbit \( v_i^H \) can
	be deduced from the value of \( u_i \). Since \( V \) is generated by orbits \( v_i^H \), \( i = 1, \dots, k \),
	as an abelian group, the action of \( c \) on \( V \) is completely determined by values of \( u_i \), \( i = 1, \dots, k \).

	There are at most \( n^k \) choices for the vectors \( u_i \), \( i = 1, \dots, k \),
	therefore \( |C_{\GL_d(p)}(H)| \leq n^k \). Now one can enumerate all elements of \( C_{\GL_d(p)}(H) \)
	by applying Lemma~\ref{ronyaiAlgo}.
\end{proof}
\begin{corollary}\label{fieldClAlgo}
	Let \( G_0 \leq \GL_d(p) \) and suppose that \( V \) decomposes into \( k \)
	irreducible \( [G_0, G_0] \)-modules over \( \GF(p) \).
	Then for every \( a \) dividing \( d \),
	one can find all groups of the form \( \GamL_a(p^{d/a}) \) containing \( G_0 \) in time polynomial in \( n^k \).
\end{corollary}
\begin{proof}
	Suppose that \( G_0 \leq \GamL_a(p^{d/a}) \) and let \( c \in \GL_d(p) \) be a
	generator of the center~\( Z(\GL_a(p^{d/a})) \). Since \( H = [G_0, G_0] \) lies in  \( \GL_a(p^{d/a}) \),
	the element \( c \) centralizes~\( H \). By Lemma~\ref{centrAlgo}, we can enumerate all elements of \( C_{\GL_d(p)}(H) \)
	in time polynomial in \( n^k \). By Lemma~\ref{fieldelem}, we can check in polynomial time if
	\( c \in C_{\GL_d(p)}(H) \) generates a finite field of order \( p^{d/a} \) and
	if \( G_0 \leq \GamL_a(p^{d/a}) \), where \( \GamL_a(p^{d/a}) \) corresponds to
	the finite field generated by \( c \).
\end{proof}

In order to apply algorithms for embeddings into \( \GamL_a(q) \) in the particular case of
groups from class (T), we first need a more precise analysis of the structure of \( G_0 \) in class (T).

\begin{proposition}\label{tensSL}
	Let \( G \) be an affine permutation group of rank~\( 3 \), such that \( G_0 \) stabilizes a tensor decomposition
	\( V = U \otimes W \) over \( \GF(q) \), where \( \dim U = 2 \) and \( \dim W = m \), \( m \geq 2 \).
	Then one of the following holds.
	\begin{enumerate}[\normalfont (i)]
		\item \( G_0 \leq \GamL_1(q) \);
		\item \( G_0 \) is imprimitive as a linear group;
		\item \( m = 4 \) and \( q = 2 \);
		\item \( \SL(W) \unlhd G_0 \);
		\item \( \SL(U) \unlhd G_0 \) and \( m = 3 \).
	\end{enumerate}
\end{proposition}
\begin{proof}
	Follows directly from~\cite[Theorem~3.2]{bambergRank3}. In the terminology the authors use,
	\( G_0 \) falls into one of the classes (R0), (I0), (T1)--(T6). Class (R0) corresponds to (i)
	in our statement, class (I0) to (ii), class (T4) implies (iii), classes (T1)--(T3) and (T5) satisfy the conclusion of (iv)
	and class (T6) implies (v).
\end{proof}

Note that the main part of the proposition is~(v), since the rest follows from the description
of the classes (A3)--(A5) (see Theorem~\ref{affrankthree}). We also mention that~\cite{bambergRank3} proves
a lot more about the structure of \( G_0 \), see~\cite[Theorems~3.1 and 3.2]{bambergRank3}.

\begin{lemma}\label{tensMods}
	Suppose that \( G \) lies in class (T). Set \( H = [G_0, G_0] \) and let
	\[ V = V_1 \oplus \dots \oplus V_k, \]
	be a decomposition of \( V \) into a direct sum of irreducible \( H \)-modules over \( \GF(p) \).
	Then \( k \leq 3 \).
\end{lemma}
\begin{proof}
	Let \( V = U_1 \otimes U_2 \) be a tensor decomposition preserved by \( G_0 \).
	By definition of the class (T), \( G_0 \) is primitive as a linear group
	and since \( m \geq 2 \), we have \( a(G) \geq 2 \) (see Table~\ref{graphtab}),
	and consequently \( G_0 \) does not lie in \( \GamL_1(q) \). By Proposition~\ref{tensSL},
	\( G_0 \) contains a normal subgroup \( S = \SL(U_2) \) preserving
	subspaces of the form \( u \otimes U_2 \), for \( u \in U_1 \) (in the notation of Proposition~\ref{tensSL},
	\( U_2 = W \) in case (iv) and \( U_2 = U \) in case (v)). Moreover, \( \dim U_1 \leq 3 \).
	Since \( n \) is large enough, we may assume that \( S \) is neither \( \SL_2(2) \) nor \( \SL_2(3) \).
	Hence \( S \) is perfect and we have \( S \leq H \).

	Now, \( S \) acts transitively on \( U_2 \), so \( u \otimes U_2 \) is an irreducible
	\( S \)-module of \( V \) for every nonzero \( u \). Hence \( V \) admits the following
	decomposition into irreducible \( S \)-modules:
	\[ V = \sum_{i = 1}^l u_i \otimes U_2, \]
	where \( u_1, \dots, u_l \) is a basis of \( U_1 \), in particular, \( l \leq 3 \).
	Since \( S \leq H \) we immediately obtain \( k \leq 3 \).
\end{proof}
\begin{corollary}\label{allgl}
	Suppose that \( G \) lies in class (T).
	Then all possibilities for \( a \) and \( \GamL_a(p^{d/a}) \) such that \( G_0 \leq \GamL_a(p^{d/a}) \)
	can be found in polynomial time.
\end{corollary}
\begin{proof}
	Follows from Corollary~\ref{fieldClAlgo} and Lemma~\ref{tensMods}.
\end{proof}

Next we show that we can check if a given basis is a tensor basis in polynomial time.

% TODO: rank 3 here?
\begin{lemma}\label{checkbase}
	Suppose that the field \( \GF(q) \) is given, and let \( v_{ij} \), \( i = 1, 2 \), \( j = 1, \dots, m \)
	be a \( \GF(q) \)-basis of \( V \). Then in polynomial time it is possible to check
	that \( G_0 \) preserves a tensor decomposition \( V = U \otimes W \), and that
	\( v_{ij} = u_i \otimes w_j \), where \( u_1, u_2 \) is some basis of \( U \) and \( w_1, \dots, w_m \) is some basis of \( W \).
\end{lemma}
\begin{proof}
	Construct a formal tensor product \( V' = U' \otimes W' \), where \( U' \) has basis \( u_1, u_2 \) and \( W' \)
	has basis \( w_1, \dots, w_m \) over \( \GF(q) \). Let \( \phi : V \to V' \) be an isomorphism defined
	by the rule \( \phi(v_{ij}) = u_i \otimes w_j \), \( i = 1, 2 \), \( j = 1, \dots, m \). Clearly we
	can define the action of \( G_0 \) on \( V' \) by \( \phi(v)^g = \phi(v^g) \), where \( v \in V \), \( g \in G_0 \).

	To finish the proof of the claim it suffices to check that \( G_0 \) preserves a tensor decomposition \( V' = U' \otimes W' \).
	That amounts to checking that for all \( u' \in U' \), \( w' \in W' \) and \( g \in G_0 \) we have
	\( (u' \otimes w')^g = u'' \otimes w'' \) for some \( u'' \in U' \) and \( w'' \in W' \). 
	Notice that we only need to check that for elements \( g \) taken from some generating set of \( G_0 \).
	The size of that generating set can be polynomially bounded in terms of \( n \), and there are at most \( n^4 \)
	choices for \( u', w', u'', w'' \). Thus this check can be performed in polynomial time and the claim follows.
\end{proof}

% TODO: understand this todo
% TODO: check that field acts appropriately + mention that in the following lemma and algo

\begin{lemma}\label{tens2Clgens}
	Let \( G \) lie in class (T), and suppose that the associated field and 
	some tensor basis are known. Then one can compute generators of \( G^{(2)} \)
	in polynomial time.
\end{lemma}
\begin{proof}
	Let \( G_0 \leq \GamL_a(q) = \GamL(V) \), where \( V = U \otimes W \),
	and let \( u_i \otimes w_j \), \( i = 1, 2 \), \( j = 1, \dots, m \), be the given tensor basis.
	Let \( g_1, \dots, g_s \) and \( h_1, \dots, h_r \) be generators of \( \GL(U) \) and \( \GL(W) \)
	respectively (take, for example, diagonal matrices and transvections, so \( s \) and \( r \) are polynomial in \( n \)).
	The following linear transformations, given by their action on the basis, preserve the tensor product:
	\begin{equation*}
	\begin{split}
		u_i \otimes w_j \mapsto u_i^{g_k} \otimes w_j,\; k = 1, \dots, s,\\
		u_i \otimes w_j \mapsto u_i \otimes w_j^{h_k},\; k = 1, \dots, r.
	\end{split}
	\end{equation*}
	Clearly these transformations generate the central product \( \GL(U) \circ \GL(W) \).

	Since the field \( \GF(q) \) is known, we can add the following semilinear map to our generating set:
	\[
		\sum_{i,j} c_{ij}(u_i \otimes w_j) \mapsto \sum_{i,j} c_{ij}^p(u_i \otimes w_j),
	\]
	where \( p \) is the characteristic of the field. This element generates field automorphisms,
	hence we obtained generators for the group \( (\GL(U) \circ \GL(W)) \rtimes \Aut(\GF(q)) \).
	
	Finally, if \( m = 2 \), add the following linear map, given by its action on the tensor basis:
	\[ u_i \otimes w_j \mapsto u_j \otimes w_i. \]
	This element generates the additional automorphism of order \( 2 \). Since the number of generators
	obtained is bounded polynomially in terms of \( n \) and all computations were performed in polynomial time,
	the claim is proved.
\end{proof}

\begin{lemma}\label{slwalgo}
	Suppose that \( G \) lies in class (T) and preserves a tensor decomposition \( V = U \otimes W \),
	where \( \dim U = 2 \), \( \dim W = m \) and \( m \geq 4 \). Suppose also that the respective field is known.
	Then \( \SL(W) \unlhd G_0 \) and it is possible to find that subgroup in polynomial time.
\end{lemma}
\begin{proof}
	The fact that \( \SL(W) \unlhd G_0 \) follows from Proposition~\ref{tensSL}
	(or directly from the description of classes (A3) and (A4)).
	In order to find \( \SL(W) \) explicitly, first compute the second derived subgroup \( H \) of \( G_0 \).
	Set \( F = \GL(q) \). Since \( G_0 \leq (\GL(U) \circ \GL(W)) \rtimes \Aut(F) \), we have \( H \leq \SL(U) \circ \SL(W) \).
	Now, \( \SL(W) \) is precisely the preimage of \( O_{\Sigma}(H/(H \cap F^{\times})) \) in \( H \),
	where \( \Sigma = \{ \operatorname{PSL}_m(q) \} \).
\end{proof}

The following lemma about arbitrary (not necessarily rank~3) permutation groups allows one to ``shift''
subsets of the permutation domain away from other subsets.

\begin{lemma}\label{shiftlem}
	Let \( G \) be a permutation group on \( \Omega \), and let \( \Gamma, \Delta \subseteq \Omega \).
	Then one can compute \( g \in G \) such that \( \Gamma^g \not\subseteq \Delta \) in polynomial time,
	or determine that there is no such \( g \).
\end{lemma}
\begin{proof}
	Let \( |\Omega| = n \) and set \( \Gamma = \{ \gamma_1, \dots, \gamma_k \} \), \( k \leq n \).
	One can compute orbits \( \gamma_i^G \), \( i = 1, \dots, k \), in polynomial time.
	Clearly all orbits lie inside \( \Delta \) if and only if \( \Gamma^g \subseteq \Delta \) for all \( g \in G \),
	so we can test if the required \( g \) exists in polynomial time. Now, if there exists an orbit \( \gamma_i^G \)
	such that \( \gamma_i^G \not\subseteq \Delta \), then take \( \gamma \in \Delta \setminus \gamma_i^G \).
	One can now find \( g \in G \) with \( \gamma = \gamma_i^g \) in polynomial time,
	and \( \Gamma^g \not\subseteq \Delta \).
\end{proof}

\begin{proposition}\label{tensAlgoProof}
	If \( G \) belongs to (T), then we can compute \( G^{(2)} \) in polynomial time.
\end{proposition}
\begin{proof}
	Suppose that \( G_0 \) preserves a tensor decomposition \( V = U \otimes W \)
	over the field \( \GF(q) \) (notice that \( q \) is not necessarily the largest
	power of \( p \) such that \( G_0 \leq \GamL_a(q) \) for some \( a \)). Without loss
	of generality we may assume that \( \dim U = 2 \) and \( \dim W = m \) for some \( m \).
	We claim that in polynomial time it is possible to find the field \( \GF(q) \)
	and a \( \GF(q) \)-basis of \( V \) of the form \( u_i \otimes w_j \), \( i = 1, 2 \), \( j = 1, \dots, m \),
	where the pair \( u_1, u_2 \) is some basis of \( U \) and \( w_1, \dots, w_m \) is some basis of \( W \).
	Given such a basis we can compute \( G^{(2)} \) in polynomial time by Lemma~\ref{tens2Clgens}.

	Set \( q \) to be the largest power of \( p \) dividing the second subdegree,
	and notice that it is the order of the field required, see Table~\ref{subtab}.
	Also let \( m \) be an integer such that \( n = q^{2m} \), and notice that it is equal to the dimension of \( W \)
	over \( \GF(q) \). Observe that by Corollary~\ref{fieldClAlgo} and Lemma~\ref{tensMods}, we can find all possibilities
	for the field \( \GF(q) \) in polynomial time.

	Now, if \( m \leq 3 \), there are at most \( n^6 \) bases of \( V \) over \( \GF(q) \).
	Using Lemma~\ref{checkbase}, we can check all such bases by brute force in polynomial time, so we are done in this case.
	From now on we may assume that \( m \geq 4 \).

	Choose two vectors \( v_1, v_2 \in V \) and assume they have the form
	\( v_1 = u_1 \otimes w \), \( v_2 = u_2 \otimes w \) for
	some linearly independent (over \( \GF(q) \)) vectors \( u_1, u_2 \in U \). Then \( v_1 \) and \( v_2 \)
	span the subspace \( P = U \otimes w \).
	By Lemma~\ref{slwalgo}, we can compute \( \SL(W) \leq G_0 \) in polynomial time.
	If \( g \in \SL(W) \), then \( P^g = U \otimes w^g \), therefore we have either \( P = P^g \) or \( P \cap P^g = 0 \)
	depending on whether \( w \) and \( w^g \) are linearly dependent or not. Since \( \SL(W) \) acts transitively
	on nonzero vectors of \( W \), there exists \( g \in \SL(W) \) such that \( P \cap P^g = 0 \), 
	and we can find such \( g \) by applying Lemma~\ref{shiftlem} with \( \Gamma = \Delta = P \).
	In general, if \( Q = P^{g_1} \oplus P^{g_2} \oplus \dots \oplus P^{g_k} \),
	\( g_1, \dots, g_k \in \SL(W) \), \( g_1 = 1 \), is a proper subspace of \( V \), then for some \( g \in \SL(W) \) the subspace
	\( P^g \) does not lie in \( Q \) and hence \( Q \cap P^g = 0 \). This allows us to add another direct summand
	to \( Q \) and to proceed further. After the final step we will end up with precisely \( m \) direct summands.

	Now, since \( V = P^{g_1} \oplus \dots \oplus P^{g_m} \), where \( g_1 = 1 \),
	the set \( v_1^{g_1}, v_2^{g_1}, \dots, v_1^{g_m}, v_2^{g_m} \) is a basis for \( V \).
	Notice that \( v_i^{g_j} = u_i \otimes w^{g_j} \) and that vectors \( w^{g_1}, \dots, w^{g_m} \) form a basis for \( W \).
	Setting \( w_j = w^{g_j} \), \( j = 1, \dots, m \), we obtain that \( v_i^{g_j} = u_i \otimes w_j \) is a required tensor basis for \( V \).
	Since there are at most \( n^2 \) choices for vectors \( v_1, v_2 \) and since we can check if the basis obtained is correct
	using Lemma~\ref{checkbase}, the claim follows.
\end{proof}

We provide a summary of the algorithm for class (T) in the following figure (Algorithm~\ref{tensorAlgo}).
Notice that in contrast to the algorithm for class (S), this algorithm may not output ``failure'' if
the input group fails to lie in (T).

\begin{algorithm}[p]
	\SetAlgoLined
	\KwIn{generators of an affine rank~3 group \( G \) of sufficiently large degree \( n = p^d \)}
	\KwOut{generators of \( G^{(2)} \) if \( G \) lies in (T)}
	\medskip

	Check that \( G \) is affine and construct \( G_0 \)\;
	\smallskip

	Set \( q \) to be the largest power of \( p \) dividing the second subdegree\;
	Find an integer \( m \) such that \( n = q^{2m} \)\;
	\smallskip

	\For(\tcp*[f]{Corollary~\ref{allgl}}){all embeddings of \( G_0 \) into \( \GamL_{2m}(q) \)}{
		Find the associated field \( \GF(q) \)\;
		\If{\( m \leq 3 \)}{
			\For{all bases of \( V \) over \( \GF(q) \)}{
				\If(\tcp*[f]{Lemma~\ref{checkbase}}){it is a tensor basis}{
					\Return{\( G^{(2)} \)}\tcp*[f]{Lemma~\ref{tens2Clgens}}
					}
				}
				\tcp{There is no appropriate basis}
				\Return{failure}
			}
		\smallskip

		Find the subgroup \( \SL(W) \leq G_0 \) \tcp*{Lemma~\ref{slwalgo}}
		\For{all \( v_1, v_2 \in V \setminus 0 \)}{
			Compute the span \( P = \langle v_1, v_2 \rangle \)\;
			Set \( g_1 = 1 \)\;
			\For{\( i = 2, \dots, m \)}{
				Find \( g_i \in \SL(W) \) such that\\
				\( P^{g_1} \oplus \dots \oplus P^{g_{i-1}} \cap P^{g_i} = 0 \) \tcp*{Lemma~\ref{shiftlem}}
				}
			\If(\tcp*[f]{Lemma~\ref{checkbase}}){\( v_1^{g_1}, v_2^{g_1}, \dots, v_1^{g_m}, v_2^{g_m} \) is a tensor basis}{
				\Return{\( G^{(2)} \)}\tcp*[f]{Lemma~\ref{tens2Clgens}}
				}
			}
		}
		\smallskip

		\tcp{No embedding gives a tensor product basis}
		\Return{failure}
		\caption{Tensor product case\label{tensorAlgo}}
\end{algorithm}

\subsubsection{Quadratic form case}\label{subsubsQuadratic}

Let \( V \) be a vector space over \( \GF(q) \).
Recall that a mapping \( \kappa : V \to \GF(q) \) is called a quadratic form, if
\( \kappa(\lambda v) = \lambda^2 \kappa(v) \) for all \( \lambda \in \GF(q) \) and \( v \in V \),
and the mapping \( f : V \times V \to \GF(q) \) defined by
\[ f(x, y) = \kappa(x+y) - \kappa(x) - \kappa(y) \]
is a bilinear form (the bilinear form associated to \( \kappa \)). The quadratic
form \( \kappa \) is non-degenerate if the radical \( \operatorname{rad}(f) \) is trivial, where
\[ \operatorname{rad}(f) = \{ x \in V \mid f(x, y) = 0 \text{ for all } y \in V \}. \]
There are two types of quadratic forms on vector spaces of even dimension: the \( + \) type and \( - \) type;
the precise definition will not be important for us, so we refer the interested reader to~\cite[Section~2.5]{kleidman} for details.

If \( G \) lies in class (Q), then \( G = V \rtimes G_0 \), where \( V \) is a \( 2m \)-dimensional
vector space over \( \GF(q) \), and \( G_0 \leq \GamL_{2m}(q) \). The group \( G_0 \)
acts on a non-degenerate quadratic form \( \kappa : V \to \GF(q) \) of type \( \epsilon = \pm \)
by semisimiliarities, i.e. \( \kappa(v^g) = \lambda\kappa(v)^\alpha \) for all \( v \in V \),
where \( \lambda \in \GF(q) \), \( \alpha \in \Aut(\GF(q)) \) depend only on \( g \in G_0 \).
Also, \( G_0 \) contains either \( \SU_m(q) \) or \( \Omega^{\epsilon}_{2m}(q) \) as a normal subgroup,
depending on whether \( G \) lies in (A6) or (A7).

There are exactly two orbits of \( G_0 \) on nonzero vectors, one orbit consisting of isotropic vectors
(i.e. those \( v \) with \( \kappa(v) = 0 \)), and another consisting of non-isotropic vectors.
The number of non-isotropic vectors is divisible by \( q \), while the number of isotropic vectors is not
(see, for instance,~\cite[Section~3.6 and formula~(3.27)]{wilson}),
in particular, one can determine isotropic vectors in polynomial time just by taking an orbit of appropriate size.

In order to compute the \( 2 \)-closure, we will show that either \( G_0 \) lies in \( \GamL_4(q) \),
and \( G^{(2)} \) can be found by brute-force, or the quadratic form \( \kappa \) can be recovered,
and \( G^{(2)} \) can be constructed explicitly. We start with preliminary lemmata which will allow us
to recover the field and find embeddings of \( G_0 \) into \( \GamL_a(q) \).

Recall that a prime \( r \) dividing \( x^k - 1 \), \( k \geq 1 \), is called a \emph{primitive prime divisor}
if \( r \) does not divide \( x^m - 1 \) for \( 1 \leq m < k \). We remind a special case of Zsigmondy's theorem.

\begin{lemma}[\cite{zsigmondy}]\label{zsigm}
	The number \( x^k - 1 \), where \( x, k \geq 2 \), has a primitive prime divisor
	except when \( x = 2 \) and \( k = 6 \), or \( x+1 \) is a power of two and \( k = 2 \).
\end{lemma}

\begin{lemma}\label{qcentr}
	Let \( G = V \rtimes G_0 \) lie in class (Q), and let \( S \unlhd G_0 \), where \( S = \SU_m(q) \) or \( \Omega_{2m}^{\pm}(q) \),
	\( m \geq 2 \),	depending on whether \( G \) lies in (A6) or (A7) respectively.
	Suppose that \( |V| > 2^{12} \).
	Then \( V \) decomposes into a direct sum of \( k \) irreducible \( S \)-modules over \( \GF(p) \),
	where \( k \leq 2 \) when \( m = 2 \), and \( k = 1 \) when \( m > 2 \).
\end{lemma}
\begin{proof}
	Set \( F = \GF(q^2) \) in the case \( S = \SU_m(q) \) and \( F = \GF(q) \) in the case \( S = \Omega_{2m}^{\pm}(q) \).
	By \cite[Proposition~2.10.6]{kleidman}, \( S \) acts irreducibly on \( V \) over \( F \), in particular,
	\( F^{\times}S \) acts irreducibly on \( V \) over \( \GF(p) \). Since \( S \) is a normal subgroup of
	\( F^{\times}S \), the module \( V \) is completely reducible and decomposes
	into a direct sum of isomorphic irreducible \( S \)-modules over \( \GF(p) \):
	\[
		V = U_1 \oplus \dots \oplus U_k.
	\]
	Clearly \( k = \frac{\dim V}{\dim U_1} \), where dimension is computed over \( \GF(p) \).
	We wish to derive an upper bound on \( k \).

	Set \( q = p^r \), where \( p \) is a prime. Since \( |F : \GF(p)| \leq 2r \), we obtain \( k \leq 2r \).
	Set \( s = \dim U_1 \).
	As \( S \) acts faithfully on \( V \), it embeds into the following direct product
	\[ \GL(U_1) \times \dots \times \GL(U_k) \simeq \GL_s(p)^k, \]
	and, in particular, \( |S| \) divides \( |\GL_s(p)|^k \).
	Recall that the order of \( \GL_s(p) \) is equal to \( p^{s(s-1)/2}\prod_{i=1}^s (p^s-1) \).

	We consider the case \( S = \SU_m(q) \) first. By \cite[Table~2.1.C]{kleidman}, the order of \( S \)
	is divisible by \( q^m-(-1)^m \) and \( q^{m-1}-(-1)^{m-1} \).
	If \( m \) is even, then \( q^{m-1}-(-1)^{m-1} = p^{(m-1)r}+1 \).
	By Lemma~\ref{zsigm}, the number \( p^{2(m-1)r} - 1 \) has a primitive prime divisor unless
	\( 2(m-1)r = 2 \), or \( p = 2 \) and \( 2(m-1)r = 6 \), where the latter situation can be discarded
	since \( |V| = p^{2mr} \) is large enough. Now, \( 2(m-1)r = 2 \) implies \( r = 1 \) and \( m = 2 \); we can
	use the previously obtained bound \( k \leq 2r = 2 \) in this case. Thus we can assume that
	\( p^{2(m-1)r} - 1 \) has a primitive prime divisor.

	Notice that the primitive prime divisor of \( p^{2(m-1)r} - 1 \)
	must divide \( p^{(m-1)r}+1 \) and, hence, must divide \( |S| \).  As \( |S| \) divides \( |\GL_s(p)|^k \),
	this prime divisor must divide one of the numbers \( p^i-1 \), where \( s = 1, \dots s \), which implies \( 2(m-1)r \leq s \).
	Recall that \( \dim V = 2mr \), and therefore \( k = \frac{\dim V}{\dim U_1} = \frac{2mr}{s} \leq \frac{2mr}{2(m-1)r} = \frac{m}{m-1} \).
	It can be readily seen that \( k \leq 2 \) if \( m = 2 \), and \( k \leq 1 \) if \( m > 2 \), as required.

	If \( m \) is odd, then	\( q^m-(-1)^m = p^{mr}+1 \) divides \( |S| \). Again, since \( 2mr > 2 \)
	(and the case \( p = 2 \), \( 2mr = 6 \) is excluded by \( |V| > 2^6 \))
	the number \( p^{2mr} - 1 \) has a primitive prime divisor dividing \( p^{mr}+1 \), and hence dividing \( |S| \).
	We obtain \( 2mr \leq s \), and therefore \( k \leq \frac{2mr}{2mr} = 1 \) in this case.
	
	We now consider the case \( S = \Omega_{2m}^{\pm}(q) \). By \cite[Table~2.1.C]{kleidman}, the order of \( S \)
	is divisible by \( q^{2m-2}-1 = p^{2(m-1)r}-1 \). If \( r = 1 \), then \( S \) is defined over \( \GF(p) \)
	and is already irreducible on \( V \), hence we may assume \( r \geq 2 \). Now \( 2(m-1)r > 2 \),
	and \( p^{2(m-1)r}-1 \) has a primitive prime divisor.
	Proceeding as in the previous case we derive \( 2(m-1)r \leq s \). Recall that \( \dim V = 2mr \), since \( |F| = p^r \),
	so \( k = \frac{2mr}{s} \leq \frac{2mr}{2(m-1)r} = \frac{m}{m-1} \). Thus \( k \leq 2 \) if \( m = 2 \),
	and \( k \leq 1 \) if \( m > 2 \), as claimed.
\end{proof}

We mention that the case \( k = 2 \) can indeed happen, for instance, \( \SU_2(q) \)
has an invariant two-dimensional \( \GF(q) \)-subspace in its action on \( \GF(q^2)^2 \)
(see an example at the beginning of Section~3.6 in~\cite{wilson} for more details).

\begin{corollary}\label{qfield}
	Let \( G \) lie in class (Q). Then all possibilities for \( a \) and \( \GamL_a(p^{d/a}) \) such that \( G_0 \leq \GamL_a(p^{d/a}) \)
	can be found in polynomial time.
\end{corollary}
\begin{proof}
	% TODO: a better reference for that?
	% SU is in Grove, Th.11.22 or Aschbacher 43.12
	% Structure of Omega_4 is given in Prop. 2.9.1 in K-L, for SL and PSL see Grove, Th.1.9
	Since \( n \) is large enough, \( \SU_m(q) \) and \( \Omega_{2m}^{\pm}(q) \) are perfect,
	see, for instance,~\cite[Proposition~2.9.2 and remark after Proposition~2.10.6]{kleidman}.
	The claim now follows from Lemma~\ref{qcentr} and Corollary~\ref{fieldClAlgo}.
\end{proof}

\begin{proposition}\label{simpRecog}
	If \( G \) is a simple unitary or orthogonal permutation group, then one can compute
	a pair of elements generating \( G \) in polynomial time.
\end{proposition}
\begin{proof}
	Let \( n \) be the degree of \( G \). If \( |G| \leq n^8 \), one can find 
	required generators by brute-force.
	If \( |G| > n^8 \), one can compute an isomorphism from \( G \)
	into a classical group given in its natural matrix representation~\cite[Theorem~9.3]{kantorSylow}.
	Now, explicit pairs of matrices generating unitary and orthogonal groups were found in~\cite{taylorMatrix}
	and~\cite{rylandsOrth}, and the claim follows.
\end{proof}

We note that pairs of generators for all Chevalley groups were computed by Steinberg~\cite{steinbergGens}
in terms of the root structure, but we require specific matrices to be able to lift them to generators of the original permutation group.

\begin{lemma}\label{gens}
	Let \( G \) lie in class (Q). Then either \( G_0 \leq \GamL_4(q) \) or \( G_0 \) is 6-generated
	and one can find such a set of generators in polynomial time.
\end{lemma}
\begin{proof}
	Let \( S \unlhd G_0 \), where \( S \) is \( \SU_m(q) \) or \( \Omega_{2m}^{\epsilon}(q) \), depending on whether
	\( G \) lies in (A6) or (A7). By Lemma~\ref{qcentr}, either \( m = 2 \) or \( S \) acts irreducibly
	on the underlying vector space. In the first case \( a(G) \leq 4 \), hence \( G_0 \leq \GamL_4(q) \) (see Table~\ref{graphtab}).
	In the second case \( S \) acts irreducibly, hence the centralizer \( C = C_{\GL(V)}(S) \)
	is a multiplicative group of a finite field, and thus cyclic.

	To find six generators of \( G_0 \) in polynomial time,
	observe that \( S \) can be found in polynomial time by computing a sufficient term of derived series.
	The centralizer \( C \) can be found in polynomial time by Lemma~\ref{ronyaiAlgo},
	and hence the associated field \( F \) can be recovered. Since the size of \( C \) is polynomially bounded,
	we can compute \( S \cap C \) in polynomial time. The simple group \( \overline{S} = S/(S \cap C) \) acts faithfully
	on the set of lines of the vector space \( V \) over \( F \), and therefore possesses a faithful permutation representation
	of degree less than \( n \). One can now apply Proposition~\ref{simpRecog} to find a pair of generators.

	Note that \( S(G_0 \cap C)/(G_0 \cap C) \simeq \overline{S} \). Now,
	\[ S(G_0 \cap C)/(G_0 \cap C) \unlhd G_0/(G_0 \cap C) \simeq G_0C/C \leq \Aut(S) \leq \Aut(\overline{S}), \]
	where the last embedding follows since \( S \) is perfect. Therefore \( G_0/(S(G_0 \cap C)) \) is a subgroup of the outer automorphism
	group of \( \overline{S} \). It is 3-generated by Lemma~\ref{almsimpGens}, and its size can be polynomially bounded in terms of \( n \)
	(see~\cite[Table~5]{atlas}),
	hence we can find three generators of \( G_0/(S(G_0 \cap C)) \) in polynomial time by brute-force. 
	The group \( G_0 \cap C \) is cyclic so we can easily find its generator, and we found two generators
	of \( \overline{S} = S/(S \cap C) \) earlier, so by taking preimages we find three generators of \( S(G_0 \cap C) \).
	By taking preimages once more, we obtain six generating elements of \( G_0 \), as claimed.
\end{proof}

The following lemma will be used to derive the value of the quadratic form on non-isotropic vectors.
\begin{lemma}\label{qformRestore}
	Let \( G_0 \) be a subgroup of \( \GamL_a(q) = \GamL(V) \) preserving a quadratic form \( \kappa : V \to \GF(q) \),
	and suppose that the field \( \GF(q) \) is known.
	Suppose we are given generators \( g_1, \dots, g_6 \) of \( G_0 \), a vector \( v \in V \),
	the value \( \kappa(v) \in \GF(q) \), scalars \( \lambda_i \in \GF(q) \) and automorphisms \( \alpha_i \in \Aut(\GF(q)) \),
	\( i = 1, \dots, 6 \), such that \( \kappa(u^{g_i}) = \lambda_i\kappa(u)^{\alpha_i} \)
	for all \( u \in V \) and \( i = 1, \dots, 6 \).
	Then we can compute in polynomial time the value of \( \kappa(u) \) for all vectors \( u \) from the orbit \( v^{G_0} \).
\end{lemma}
\begin{proof}
	We will increase the number of vectors where \( \kappa \) is known step by step.
	Denote by \( \Delta_j \subseteq v^{G_0} \) the set of vectors \( u \) for which \( \kappa(u) \) is known after
	\( j \) steps of the algorithm. Since \( v \in v^{G_0} \) and we are given \( \kappa(v) \) as input,
	we set \( \Delta_0 = \{ v \} \).

	Suppose we made \( j \) steps of the algorithm. If \( \Delta_j^{g_i} \subseteq \Delta_j \)
	for all \( i = 1, \dots, 6 \), then \( \Delta_j = v^{G_0} \) and we are done.
	Otherwise, we can find some \( u \in \Delta_j \) and \( i \in \{ 1, \dots, 6 \} \)
	such that \( u^{g_i} \not\in \Delta_j \). We can compute the value of \( \kappa \) on \( u^{g_i} \), indeed
	\[ \kappa(u^{g_i}) = \lambda_i\kappa(u)^{\alpha_i}, \]
	where the right hand side can be computed since \( \kappa(u) \) is known. We set \( \Delta_{j+1} = \Delta_j \cup \{ u^{g_i} \} \),
	and proceed with the algorithm.

	Clearly this procedure will require at most \( n \) steps to stop, and since each step can be trivially performed in polynomial time,
	the claim follows.
\end{proof}

\begin{lemma}\label{qform}
	Let \( G_0 \) be a subgroup of \( \GamL_a(q) = \GamL(V) \) and suppose that the field \( \GF(q) \) is known.
	Given a function \( \kappa : V \to \GF(q) \) we can check in polynomial time if \( \kappa \) is
	a non-degenerate quadratic form	and if \( G_0 \) acts on \( \kappa \) by semisimiliarities.
	Moreover, we can find the associated bilinear form in polynomial time.
\end{lemma}
\begin{proof}
	The last claim follows from the fact that \( f(u, v) = \kappa(u+v)-\kappa(u)-\kappa(v) \).
	It is easy to compute \( \operatorname{rad}(f) \) and check that \( \kappa \) is a non-degenerate quadratic form
	by straightforward algorithms. In order to check that \( G_0 \) acts on \( \kappa \) by semisimiliarities, it suffices to do that only
	for generators of \( G_0 \). If \( g \in G_0 \), it is easy to test if for all \( v \in V \) we have
	\( \kappa(v^g) = \lambda\kappa(v)^\alpha \), for some \( \lambda \in \GF(q) \) and \( \alpha \in \Aut(\GF(q)) \)
	depending only on \( g \) (indeed, we need only to try out a polynomially bounded number of scalars
	\( \lambda \) and automorphisms \( \alpha \)). This provides the required algorithm.
\end{proof}

\begin{lemma}\label{qform2Clgens}
	Let \( G \) lie in class (Q), and suppose that the associated field, bilinear and quadratic forms are known.
	Then one can compute generators of \( G^{(2)} \) in polynomial time.
\end{lemma}
\begin{proof}
	Let \( F = \GF(q) \) be the associated field. By Theorem~\ref{class}, we have \( G^{(2)} = F^{2m} \rtimes \GamO_{2m}^{\epsilon}(q) \),
	\( \epsilon = \pm \), for appropriate \( m \), so it suffices to find generators of \( \GamO_{2m}^{\epsilon}(q) \).
	Since the field and forms are known, we can construct the standard basis of the space
	in polynomial time; the usual procedure can be easily adopted into a polynomial-time algorithm,
	see, for example, the proof of~\cite[Proposition~2.5.3]{kleidman}.
	Now generators can be written down explicitly, see~\cite[Sections~2.7 and~2.8]{kleidman}.
\end{proof}

\begin{proposition}
	If \( G \) belongs to (Q), then we can find \( G^{(2)} \) in polynomial time.
\end{proposition}
\begin{proof}
	By Corollary~\ref{qfield}, we can find the required field and an overgroup \( \GamL_a(q) \).
	If \( a \leq 4 \), then \( G^{(2)} \leq \AGL_8(q) \) (see Table~\ref{graphtab}),
	so we can find the 2-closure in polynomial time by brute-force. Now we can assume that
	\( G_0 \) is not contained in \( \GamL_4(q) \), so by Lemma~\ref{gens}, \( G_0 \) is 6-generated and
	its generators \( g_1, \dots, g_6 \) can be obtained in polynomial time.

	We describe the procedure for constructing the quadratic form \( \kappa \).
	Choose an arbitrary non-isotropic vector \( v \neq 0 \),
	and choose some value for \( \kappa(v) \in \GF(q) \) (there are at most \( n \) such values).
	Now, choose scalars \( \lambda_1, \dots, \lambda_6 \in \GF(q) \) and automorphisms \( \alpha_1, \dots, \alpha_6 \in \Aut(\GF(q)) \);
	observe that there are at most polynomial number of such choices.
	Assume that \( \kappa(u^{g_i}) = \lambda_i \kappa(u)^{\alpha_i} \), \( i = 1, \dots, 6 \), for all \( u \in V \).
	Since all non-isotropic vectors are conjugate to \( v \) by elements of \( G_0 \), Lemma~\ref{qformRestore} implies that
	we can compute \( \kappa(u) \) for all non-isotropic \( u \) and hence for all \( u \in V \).

	Using Lemma~\ref{qform} we can check that \( \kappa \) is a non-degenerate quadratic form and that \( G_0 \) acts
	on it by semisimiliarities. If that is the case for at least one choice of the field, \( \kappa(v) \), scalars
	\( \lambda_i \) and automorphisms \( \alpha_i \), \( i = 1, \dots, 6 \), then \( G \) belongs to (Q).
	Clearly the described procedure will succeed if \( G \) does belong to (Q), so the first claim is proved.

	Now, in order to construct the 2-closure of \( G \),
	we restore the bilinear form associated with \( \kappa \) and apply Lemma~\ref{qform2Clgens}.
\end{proof}

We provide an overview of the algorithm in case (Q) in the following figure (Algorithm~\ref{qformAlgo}).
Notice that this algorithm may find the 2-closure correctly even if the group does not belong
to class (Q), for example, that is the case with some subgroups of \( \GamL_4(q) \) lying in class (S).

\begin{algorithm}[p]
	\SetAlgoLined
	\KwIn{generators of an affine rank~3 group \( G \) of sufficiently large degree \( n = p^d \)}
	\KwOut{generators of \( G^{(2)} \) if \( G \) lies in (Q)}
	\medskip

	Check that \( G \) is affine and construct \( G_0 \)\;
	\smallskip

	\tcp{Case when \( a(G) = 4 \)}
	Set \( L \) to be an empty list\;
	\For(\tcp*[f]{Corollary~\ref{qfield}}){all embeddings of \( G_0 \) into \( \GamL_4(q) \)}{
		Add \( G^{(2)} \cap \AGL_8(q) \) to the list \( L \)\;
		}
	\If{\( L \) is not empty}{
		\Return{the largest member of \( L \)}
		}
	\smallskip

	\tcp{Case when \( a(G) > 4 \)}
	Find elements \( g_1, \dots, g_6 \) generating \( G_0 \) \tcp*{Lemma~\ref{gens}}
	\For(\tcp*[f]{Corollary~\ref{qfield}}){all embeddings of \( G_0 \) into \( \GamL_a(q) \)}{
		Find the associated field \( \GF(q) \)\;
		Choose arbitrary non-isotropic \( v \in V \setminus 0 \)\;
		\For{all \( \gamma, \lambda_1, \dots, \lambda_6 \in \GF(q) \)}{
			\For{all automorphisms \( \alpha_1, \dots, \alpha_6 \in \Aut(\GF(q)) \)}{
				Set \( \kappa(u) = 0 \) for all isotropic vectors \( u \)\;
				Extend the function \( \kappa : V \to \GF(q) \)
				to all non-isotropic vectors using identities \( \kappa(v) = \gamma \)
				and \( \kappa(u^{g_i}) = \lambda_i\kappa(u)^{\alpha_i} \), where \( i = 1, \dots, 6 \) and \( u \in V \)
				\tcp*{Lemma~\ref{qformRestore}}
				\If(\tcp*[f]{Lemma~\ref{qform}}){\( \kappa \) is a non-degenerate quadratic form\\
								 and \( G_0 \) acts on it by semisimiliarities}{
									 Restore the associated bilinear form\;
									 \Return{\( G^{(2)} \)}\tcp*[f]{Lemma~\ref{qform2Clgens}}
					}
				}
			}
		}
		\smallskip

		\tcp{No embedding or choice of vectors, scalars or automorphisms defines a correct quadratic form}
		\Return{failure}
		\caption{Quadratic form case\label{qformAlgo}}
\end{algorithm}

Finally, we remind the overall algorithm for computing 2-closures of rank~3 groups.
If degree of the group is at most \( 3^{12} \), we find the 2-closure by brute-force.
Otherwise, we apply Algorithms~\ref{nonaffAlgo}--\ref{qformAlgo} to our group and obtain
at most four possible answers (some algorithms may fail, but at least one of them must succeed
by our classification of rank~3 groups into nonaffine, (S), (T) and (Q) cases). We choose the largest
output as our final answer. Since all of four candidate groups lie in the 2-closure and at least
one coincides with it, the algorithm works correctly.

\section{Conclusion}\label{secConcl}

It was shown that the 2-closure of rank~3 groups can be computed in polynomial time.
We would like to make a few comments regarding the complexity of the algorithm,
some implications of the methods involved and mention a few directions for future research.

Since the question of exact complexity estimates was out of scope of this paper, it is easy
to notice that our algorithms are of purely theoretical nature and their precise complexity is quite large.
The bottleneck here seems to be Algorithm~\ref{smallAlgo}, which requires at least \( O(n^{50}) \) operations in the worst case
when searching for a generating 4-tuple of elements of \( G \). It seems plausible that one can
lower this bound by designing custom algorithms for graphs from class (S), such as, for example,
the Suzuki--Tits ovoid graph and the alternating forms graph.

It should also be mentioned that in many cases we do not only construct the 2-closure,
but recognize the corresponding rank~3 graph as well, with the only exception being graphs
arising from almost simple groups (i.e.\ Theorem~\ref{class}~(iii)). The claim is clear
for Hamming graphs, bilinear forms graphs (class (T)) and affine polar graphs (class (Q)).
For graphs arising from class (S) one can use a slight modification of Algorithm~\ref{smallAlgo}.
One constructs canonic copies of ``potential'' rank~3 graphs (\( A(5, q) \), \( \VO_8^+(q) \), etc.)\ and
their automorphism groups, and then uses the embedding procedure (Lemma~\ref{embed}) to
conjugate the given rank~3 group \( G \) inside the full automorphism group of the potential graph.
It is then easy to check if \( G \) does indeed stabilize the graph, and since there are only polynomially
many potential graphs the whole recognition procedure runs in polynomial time.

Finally, an important open problem is the recognition of rank~3 graphs without using the
corresponding rank~3 group:
\medskip

\noindent
\textbf{Question.} Given a graph, is it possible to check if it is a rank~3 graph and
find an appropriate labelling in polynomial time?
\medskip

Notice that the result of Imrich and Klav\u{z}ar~\cite{imrichHamming} answers this question in the case
of a Hamming graph. It is interesting whether other classes of rank~3 graphs admit such an algorithm,
for example, the question is open even for the bilinear forms graph. Clearly if one could construct a
polynomial algorithm for recognizing (and labelling) all rank~3 graphs, it would supersede the results
presented in this paper, and moreover, would solve the graph isomorphism problem for rank~3 graphs.

One promising approach to this problem is to bound the combinatorial base number of respective graphs;
combinatorial base (also called the EP-base in~\cite{baileyBases}, and just base in~\cite{evdokimovCycl})
is a generalization of the usual notion of a base of a permutation group. If a graph has combinatorial base number \( b \)
then its full automorphism group has base number at most~\( b \), in particular, its order is bounded by \( n^b \).
This implies that many classes of rank~3 graphs cannot have bounded base since they have large automorphism groups,
yet it is quite possible that all rank~3 graphs associated with affine rank~3 groups from class (S)
may have combinatorial base number bounded by some universal constant. In~\cite{evdokimovCycl} it was shown than
Paley graphs have combinatorial base number at most~3, but it seems to be the only known case, the problem being open
even for other one-dimensional affine graphs such as Peisert graphs.

\section{Acknowledgements}

The author would like to express his gratitude to M.A.~Grechkoseeva, M.~Muzychuk, I.N.~Ponomarenko,
A.V.~Vasil'ev and E.P.~Vdovin for helpful remarks and comments on this paper.

The work is supported by Mathematical Center in Akademgorodok under agreement
No.~075-15-2019-1613 with the Ministry of Science and Higher Education of the Russian Federation.

% \clearpage
% \bibliographystyle{plain}


\begin{thebibliography}{10}

\bibitem{babaiGI}
L.~Babai.
\newblock Groups, graphs, algorithms: the graph isomorphism problem.
\newblock In {\em Pro\-ceed\-ings of the International Congress of
  Mathematicians (ICM 2018)}, volume~3, pages 3303--3320. WORLD SCIENTIFIC,
  2019.

\bibitem{babaiCanonic}
L.~Babai and E.~M. Luks.
\newblock Canonical labeling of graphs.
\newblock In {\em Proceedings of the Fifteenth Annual ACM Symposium on Theory
  of Computing}, STOC '83, page 171–183, New York, NY, USA, 1983. Association
  for Computing Machinery.

\bibitem{baileyBases}
R.~F. Bailey and P.~J. Cameron.
\newblock Base size, metric dimension and other invariants of groups and
  graphs.
\newblock {\em Bull. London Math. Soc.}, 43(2):209--242, 2011.

\bibitem{bambergRank3}
J.~Bamberg, A.~Devillers, J.~B. Fawcett, and C.~E. Praeger.
\newblock Partial linear spaces with a rank 3 affine primitive group of
  automorphisms.
\newblock {\em Journal Lond. Math. Soc.}, 104(3):1011--1084, 2021.

\bibitem{bannai}
E.~Bannai.
\newblock Maximal subgroups of low rank of finite symmetric and alternating
  groups.
\newblock {\em J. Fac. Sci. Univ. Tokyo}, 18:475--486, 1972.

\bibitem{bodlaenderWidth}
H.~L. Bodlaender.
\newblock Polynomial algorithms for graph isomorphism and chromatic index on
  partial k-trees.
\newblock {\em J. Algorithms}, 11(4):631--643, 1990.

\bibitem{brouwerDRG}
A.~E. Brouwer, A.~M. Cohen, and A.~Neumaier.
\newblock {\em Distance-regular graphs}.
\newblock Ergebnisse der Mathematik und ihrer Grenzgebiete. 3. Folge / A Series
  of Modern Surveys in Mathematics. Springer Berlin Heidelberg, 1989.

\bibitem{srgw}
A.~E. Brouwer and H.~Van~Maldeghem.
\newblock {\em Strongly Regular Graphs}.
\newblock Encyclopedia of Mathematics and its Applications. Cambridge
  University Press, 2022.

\bibitem{cameronFinsimp}
P.~J. Cameron.
\newblock Finite permutation groups and finite simple groups.
\newblock {\em Bull. London Math. Soc.}, 13(1):1--22, 1981.

\bibitem{chistovConj}
A.~Chistov, G.~Ivanyos, and M~Karpinski.
\newblock Polynomial time algorithms for modules over finite dimensional
  algebras.
\newblock In {\em Proceedings of the 1997 International Symposium on Symbolic
  and Algebraic Computation}, ISSAC '97, page 68–74, New York, NY, USA, 1997.
  Association for Computing Machinery.

\bibitem{atlas}
J.~H. Conway, R.~T. Curtis, S.~P. Norton, R.~A. Parker, and R.~A. Wilson.
\newblock {\em Atlas of finite groups: maximal subgroups and ordinary
  characters for simple groups}.
\newblock Clarendon Press, 1985.

\bibitem{voltaGens}
F.~Dalla~Volta and A.~Lucchini.
\newblock Generation of almost simple groups.
\newblock {\em J. Algebra}, 178:194 -- 223, 1995.

\bibitem{evdokimovOdd}
S.~Evdokimov and I.~Ponomarenko.
\newblock Two-closure of odd permutation group in polyno\-mial time.
\newblock {\em Discrete Mathematics}, 235(1):221 -- 232, 2001.

\bibitem{evdokimovCycl}
S.~A. Evdokimov and I.~N. Ponomarenko.
\newblock Characterization of cyclotomic schemes and normal schur rings over a
  cyclic group.
\newblock {\em St. Petersburg Math. J.}, 14(2):189--221, 2003.

\bibitem{hopcroftPlanar}
J.~E. Hopcroft and R.~E. Tarjan.
\newblock Isomorphism of planar graphs (working paper).
\newblock In R.~E. Miller, J.~W. Thatcher, and J.~D. Bohlinger, editors, {\em
  Complexity of Computer Computations. The IBM Research Symposia Series}, pages
  131--152. Springer, Boston, MA, 1972.

\bibitem{imrichHamming}
W.~Imrich and S.~Klavžar.
\newblock Recognizing {Hamming} graphs in linear time and space.
\newblock {\em Inform. Process. Lett.}, 63(2):91 -- 95, 1997.

\bibitem{kantorSylow}
W.~M. Kantor.
\newblock Sylow's theorem in polynomial time.
\newblock {\em Journal of Computer and System Sciences}, 30(3):359 -- 394,
  1985.

\bibitem{kantor}
W.~M. Kantor and R.~A. Liebler.
\newblock The rank 3 permutation representations of the finite classical
  groups.
\newblock {\em Trans. Amer. Math. Soc.}, 271:1--71, 1982.

\bibitem{kantorQuot}
W.~M. Kantor and E.~M. Luks.
\newblock Computing in quotient groups.
\newblock In {\em Proceedings of the Twenty-Second Annual ACM Symposium on
  Theory of Computing}, STOC '90, pages 524--534, New York, NY, USA, 1990.
  Association for Computing Machinery.

\bibitem{kleidman}
P.~B. Kleidman and M.~W. Liebeck.
\newblock {\em The subgroup structure of the finite classical groups}.
\newblock London Mathematical Society Lecture Note Series. Cambridge University
  Press, 1990.

\bibitem{liebeckAffine}
M.~W. Liebeck.
\newblock The affine permutation groups of rank three.
\newblock {\em Proc. London Math. Soc.}, s3-54(3):477--516, 1987.

\bibitem{liebeckRank3}
M.~W. Liebeck and J.~Saxl.
\newblock The finite primitive permutation groups of rank three.
\newblock {\em Bull. London Math. Soc.}, 18(2):165--172, 1986.

\bibitem{luksBounded}
E.~M. Luks.
\newblock Isomorphism of graphs of bounded valence can be tested in polynomial
  time.
\newblock {\em J. Comput. System Sci.}, 25(1):42--65, 1982.

\bibitem{luksPolyNorm}
E.~M. Luks and T.~Miyazaki.
\newblock Polynomial-time normalizers.
\newblock {\em {Discrete Mathematics and Theoretical Computer Science}}, Vol.
  13 no. 4(4):61--96, 2011.

\bibitem{parkerMeat}
R.~A. Parker.
\newblock The computer calculation of modular characters (the {Meat-Axe}).
\newblock In {\em Symposium on Computational Group Theory (ed. Michael
  Atkinson, Academic Press, London, 1984)}, pages 267--274, 1984.

\bibitem{ponomVasSuper}
I.~Ponomarenko and A.~Vasil'ev.
\newblock Two-closure of supersolvable permutation group in polynomial time.
\newblock {\em Comp. Complexity}, 29(5), 2020.

\bibitem{ponomHadwiger}
I.~N. Ponomarenko.
\newblock The isomorphism problem for classes of graphs closed under
  contraction.
\newblock {\em J. Math. Sci.}, 55:1621--1643, 1991.

\bibitem{ponomGI2Cl}
I.~N. Ponomarenko.
\newblock Graph isomorphism problem and 2-closed permutation groups.
\newblock {\em AAECC}, 5:9--22, 1994.

\bibitem{rylandsOrth}
L.~J. Rylands and D.~E. Taylor.
\newblock Matrix generators for the orthogonal groups.
\newblock {\em J. Symbolic Comput.}, 25(3):351--360, 1998.

\bibitem{ronyaiFA}
L.~Rónyai.
\newblock Computing the structure of finite algebras.
\newblock {\em J. Symbolic Computation}, 9(3):355--373, 1990.

\bibitem{ronyai}
L.~Rónyai.
\newblock Computing the order of centralizers in linear groups.
\newblock {\em Inform. and Comput.}, 91(2):172 -- 176, 1991.

\bibitem{seress}
{\'A}.~Seress.
\newblock {\em Permutation group algorithms}.
\newblock Cambridge Tracts in Mathematics. Cambridge University Press, 2003.

\bibitem{skresRank3}
S.~V. Skresanov.
\newblock On 2-closures of rank 3 groups.
\newblock {\em Ars Math. Contemp.}, 21(1):\#P1.08, 2021.

\bibitem{steinbergGens}
R.~Steinberg.
\newblock Generators for simple groups.
\newblock {\em Canad. J. Math}, 14:277--283, 1962.

\bibitem{taylorMatrix}
D.~E. Taylor.
\newblock Pairs of generators for matrix groups. {I}.
\newblock {\em The Cayley Bulletin}, 3:76--85, 1987.

\bibitem{vasilev32}
A.~V. Vasil'ev and D.~V. Churikov.
\newblock The 2-closure of a {$\frac{3}{2}$}-transitive group in polynomial
  time.
\newblock {\em Sib. Math. J.}, 60(2):279--290, 2019.

\bibitem{weisfeiler}
B.~Weisfeiler.
\newblock {\em On Construction and Identification of Graphs}.
\newblock Springer Lecture Notes, 558, 1976.

\bibitem{wielandt2Cl}
H.~W. Wielandt.
\newblock {\em Permutation groups through invariant relations and invariant
  functions}.
\newblock The Ohio State University, 1969.

\bibitem{wilson}
R.~Wilson.
\newblock {\em The finite simple groups}.
\newblock Graduate Texts in Mathematics. Springer London, 2009.

\bibitem{zsigmondy}
K.~Zsigmondy.
\newblock {Zur Theorie der Potenzreste}.
\newblock {\em Monatsh. für Math. u. Phys.}, 3:265--284, 1892.

\end{thebibliography}
\end{document}